%% file: main.tex
\documentclass{siamart250211}

\usepackage[utf8]{inputenc}
\usepackage[english]{babel}
\usepackage{algorithm}
\usepackage{algorithmicx}
\usepackage{algpseudocode}
\usepackage{amssymb,bm}
\usepackage{physics}
\usepackage{cite}

\newtheorem{example}{Example}

\newcommand*{\B}[1]{#1}

\newcommand*{\F}{{\mathbb{F}}}

\DeclareMathOperator{\diag}{diag}

\DeclareMathOperator{\nnz}{nnz}
\DeclareMathOperator{\rnk}{rank}
\DeclareMathOperator{\tw}{tw}

\title{Sparse LDL, LU, and Inverse Butterfly Factorization via Tree Decomposition}
\author{Edgar Solomonik\thanks{Siebel School of Computing and Data Science, University of Illinois Urbana-Champaign, Urbana, IL 61801, USA (\email{solomon2@illinois.edu})}
}
\date{April 2025}

\begin{document}

\maketitle

\begin{abstract}
While linear systems over general fields can be solved in matrix-multiplication time, the complexity of symmetric triangular factorization has received relatively little formal study.
We give dense and sparse LDL algorithms for symmetric matrices over an arbitrary field.
Both algorithms leverage pivoted (rank-revealing) LU on off-diagonal blocks of a saddle-point form of a general symmetric matrix.
For an $n\times n$ matrix, this yields an $O(n^\omega)$ dense LDL algorithm, where $n\times n$ matrix multiplication is assumed to cost $O(n^\omega)$ with $\omega>2$.
For sparse matrices whose graph has treewidth $\tau$, we provide an implicit LDL in $O(n\tau^{\omega-1})$ time, and an explicit LDL whenever the rank deficiency is $O(\tau)$.
We give analogous results for sparse LU via a standard off-diagonal embedding.
We also obtain bounds on work, storage, and parallel-depth in terms of the dense $\tau\times\tau$ kernels executed at each bag in a tree decomposition.
Finally, in the full-rank bounded-treewidth setting, we prove that $A^{-1}$ has complementary low-rank structure and admits an exact butterfly factorization with rank $O(\tau)$.
\end{abstract}

\section{Introduction}

The use of fast matrix multiplication in Gaussian elimination was first proposed alongside the first subcubic matrix multiplication algorithm in the seminal work by Strassen~\cite{strassen1969gaussian} in 1969.
Since then, numerous improvements have lowered the complexity of matrix multiplication~\cite{Pan:1984:MMF:2212,coppersmith1987matrix,le2014powers,williams2024new}, whose cost we parameterize as $O(n^\omega)$ for $n\times n$ matrices.
We assume throughout the paper that $\omega > 2$ (if $\omega=2$, the complexity of factorization algorithms considered in the paper would have additional logarithmic factors).
Five years after Strassen's paper, Bunch and Hopcroft~\cite{bunch1974triangular} provided an approach to perform LU factorization with partial pivoting in $O(n^\omega)$ time for a full-rank real or complex $n\times n$ matrix.
Based on a suggestion by Sch\"onhage, they also point out that to solve $Ax=b$, it would suffice to form the normal equations and perform Cholesky on $A^TA$, avoiding pivoting.
However, Sch\"onhage's suggestion is not applicable in the case of solving $Ax=b$ over a general field, in particular over GF(2).
A simpler recursive algorithm for LU factorization of $m\times n$ matrices in $O(mn^{\omega-1})$ cost was proposed by 
Ibarra, Moran, and Hui~\cite{IBARRA198245} in 1981.
In the case of symmetric positive definite sparse matrices whose off-diagonal nonzeros are described by a graph $G$, the treewidth, $\tw(G)$, is equal to the minimum, over elimination orderings, of the maximum number of off-diagonal nonzeros in any column of the Cholesky factor $L$.
It is straightforward to accelerate Cholesky using fast matrix multiplication, yielding an algorithm with runtime $O(n\tw(G)^{\omega-1})$.
However, in the case of symmetric indefinite matrices, we are not aware of an analogous LDL factorization algorithm with running time matching the corresponding dense and sparse Cholesky bounds.

A system of equations with a symmetric matrix may be solved via LU factorization with partial pivoting, however, the LDL factorization may be desirable, as it can be described based on elementary vertex and edge elimination operations on an undirected graph.
Further, the LDL factorization provides additional information about the matrix, such as the inertia (number of zero, positive, and negative eigenvalues), in the case of real or complex matrices.
The primary motivating application for this work is the application of LDL factorization to quantum states resulting from Clifford circuits~\cite{van2004graphical,anders2006fast,gosset2024fast,de2022fast,aaronson2004improved,dehaene2003clifford}, such as graph states~\cite{hein2006entanglement}.
These states are described by a symmetric matrix and adhere to local-equivalences~\cite{bouchet1993recognizing,van2004graphical,dehaene2003clifford}, that are fully described by LDL-like edge and vertex elimination operations over GF(2).
Further, Gosset, Grier, Kerzner, and Schaeffer provide an $O(n^\omega)$-time solver for sparse linear systems over GF(2) by reduction to Clifford circuit simulation~\cite{gosset2024fast}.
In the Clifford-circuit setting that motivates this work, the matrices are typically sparse, so the regime of greatest interest is large $n$ together with enough combinatorial structure to control fill.
Beyond preserving symmetry and providing additional information regarding the matrix, such as in the case of graph states, the LDL factorization may also be computed with less fill and work than LU (see Example~\ref{ex:1}).
In this paper, we provide algorithms for LDL factorization of sparse and dense symmetric matrices that achieve the same complexity as the best known methods for Cholesky.
Given our primary motivation is GF(2), we consider a general field and do not give focus to numerical stability.

The main challenge associated with developing a fast-matrix-multiplication-based algorithm for sparse LDL is that vertices (rows and columns of the matrix) may not always be eliminated.
Given a tree decomposition of the graph $G$~\cite{robertson1986graph}, we would like to eliminate all vertices contained exclusively in the leaves of the tree, but cannot do so naively.
We propose an approach that first eliminates as many vertices as possible within the leaves of a tree decomposition, resulting in a saddle-point system of equations~\cite{benzi2005numerical}.
Then, our algorithm leverages a variant of the null-space method~\cite{antoniou2007practical} (see Benzi, Golub, and Liesen~\cite{benzi2005numerical} for further references) to efficiently compute the LDL decomposition of such saddle-point systems and maintain appropriately bounded fill.
Our approach is closely related to an algorithm for the null-space method proposed by Schilders~\cite{schilders2009solution}, but avoids orthogonal factorization (again not possible over GF(2)) and explicitly constructs an LDL decomposition.

Beyond the use of the null-space method to compute a sparse LDL, we introduce an implicit LDL decomposition based on a primitive we refer to as vertex (row/column) peeling.
With this primitive, we can define the LDL based on elementary transformations with bounded fill.
Peeling is applied only when the factorized matrix is rank-deficient, and in general we obtain the explicit form of the LDL for a maximal-size, full-rank, square block of a permutation of the matrix.
For a general sparse rectangular matrix $B$, the same symmetric machinery also yields LU through the off-diagonal embedding,
\[
A_B=\begin{bmatrix} 0 & B^H\\ B & 0\end{bmatrix}.
\]

Finally, we show that tree decomposition of an invertible matrix $A$, also implies that, with appropriate row/column ordering, $A^{-1}$ has low-rank off-diagonal blocks.
We leverage the LDL decomposition to show this result.
Further, we demonstrate that $A^{-1}$ admits a butterfly factorization with rank $\tau$ ($A^{-1}$ may be written as a product of $O(\log n)$ block-sparse matrices, each with $O(n/\tau)$ blocks of dimension at most $O(\tau)$).

In Section~\ref{sec:bg}, we present LDL factorization in terms of elementary graph operations, and review prior work on solving saddle point systems of equations.
In the subsequent sections, we present our contributions,
\begin{itemize}
\item in Section~\ref{sec:nsldl}, we present a new algorithm for LDL factorization of saddle point systems, adapting Schilders' null-space method so as to reduce the problem to rectangular LU and dense LDL factorization of blocks,
\item in Section~\ref{sec:app_ldl_strassen}, we give a dense LDL algorithm with $O(n^\omega)$ arithmetic complexity for $n\times n$ matrices and the same saddle-point/LU reduction in low-rank recursive cases,
\item our main contribution, presented in Section~\ref{sec:sparse_ldl}, combines these ingredients in an implicit sparse LDL factorization with cost $O(n\tau^{\omega-1})$, where $\tau$ is the treewidth of the graph associated with the matrix; the same bound also gives sparse LU through the off-diagonal embedding above,
\item in Section~\ref{sec:inverse_butterfly}, we prove complementary rank bounds for $A^{-1}$ and derive an exact butterfly factorization on the balanced binary tree hierarchy obtained from a tree decomposition.
\end{itemize}
While these matrix factorization algorithms may be combined with some forms of pivoting, quantifying stability remains an open question for all of them.

\section{Background}
\label{sec:bg}

\subsection{Notation and Algebraic Setting}

We use $a_{ij}$ and $A[i,j]$ to refer to the element in row $i$ and column $j$ of matrix $A$.
Further, $a_{:,i}$ and $A[:,i]$ refer to the $i$th column of $A$, while $A[i:j,:]$ refer to rows $i$ through $j$.
For a vector $u$ of dimension $n$, the notation $u[:l]$ and $u[m:]$ is short-hand notation for $u[1:l]$ and $u[m:n]$.
We write $X\sqcup Y$ for the disjoint union of sets $X$ and $Y$.

We present LDL decomposition for matrices over a general field $\mathbb F$ with identity elements $0,1$.
Throughout, either $a^*=a$ (the main arbitrary-field setting), or $a\mapsto a^*$ is a fixed involution on $\mathbb F$ (for example, the standard conjugation on a quadratic extension).
We assume the given matrix $A$ satisfies the symmetry relation $a_{ij}=a_{ji}^*$.
The only property of $*$ used by the elimination identities below is that this symmetry is preserved by the Schur complements arising from nonsingular $1\times 1$ and $2\times 2$ pivots.
It is straightforward to extend the algorithms and analysis in this paper to the skew-symmetric case, i.e., $a_{ij}=-a_{ji}$ or $a_{ij}=-a_{ji}^*$.
In such cases, the $D$ matrix in LDL satisfies $D^H=-D$ and consists only of anti-diagonal $2\times 2$ blocks.
Hence, we may apply LDL factorization to, e.g., a complex Hermitian or skew-Hermitian matrix, a real symmetric or skew-symmetric matrix, or a symmetric matrix over GF(2).
From here on, we assume a fixed field $\mathbb F$ and involution, and refer to any matrix satisfying the symmetry relation as $H$-symmetric, with the notation $B=A^H$ implying $b_{ij}=a_{ji}^*$.

\subsection{LDL Decomposition}
\label{subsec:ldl_def}

The LDL decomposition of an $H$-symmetric matrix $A\in\mathbb{F}^{n\times n}$ is
\[P^TAP=LDL^H,\]
where $P$ is a permutation matrix, $L$ is lower triangular and unit-diagonal, and $D$ is block diagonal with blocks that are either $1\times 1$ or $2\times 2$ and anti-diagonal.
In the symmetric/Hermitian case, $D=D^H$.
For the skew-symmetric or skew-Hermitian variant mentioned above, the same elimination pattern gives $D^H=-D$ and only anti-diagonal $2\times 2$ blocks.
When $A$ is not of full rank, the last $n-\rank(A)$ columns of $L$ are redundant, and we may obtain a reduced LDL decomposition with $L\in\mathbb{F}^{n\times\rank(A)}$ and lower-trapezoidal.

The LDL decomposition may be expressed in terms of graph operations, namely vertex and edge eliminations.
First, we associate the off-diagonal entries of $A$ with edges of a graph $G=(V,E)$, $V=\{1,\ldots, n\}$.
A vertex elimination of a vertex $i$ means first permuting $i$ to the leading position and then performing the update below; equivalently, the formulas below describe the generic operation after such a permutation.
A vertex elimination on vertex $1$ is possible if $a_{11}\neq 0$ and corresponds to computing,
\[\begin{bmatrix} 0 & 0 \\ 0 & S \end{bmatrix} = A- a_{:,1}a_{11}^{-1}a_{:,1}^H,\]
where the $(n-1) \times (n-1)$ matrix $S$ is the Schur complement of $a_{11}$ in $A$.
Each successive vertex elimination defines a successive column of $L$ ($a_{:,1}a_{11}^{-1}$) and a diagonal entry of $D$ ($a_{11}$).
A Cholesky factorization may be formed from $n$ vertex eliminations for any complex Hermitian positive definite matrix.
In LDL, we must also handle scenarios where all available diagonal entries are zero.
If $a_{11}=0$, an edge elimination on the first two vertices is possible if $a_{21}\neq 0$.
Again, this should be understood after permuting the chosen edge to the leading $2\times 2$ block.
Edge elimination corresponds to forming the $(n-2)$-dimensional square Schur complement matrix $S$, where
\[\begin{bmatrix} 0 & 0 \\ 0 & S \end{bmatrix} = A- 
\begin{bmatrix} a_{:,1} & a_{:,2} \end{bmatrix}
\begin{bmatrix} a_{11} & a_{12} \\ a_{21} & a_{22} \end{bmatrix}^{-1}
\begin{bmatrix} a_{:,1} & a_{:,2} \end{bmatrix}^H.\]
Note that if $a_{11}=0$, the inverse $B=\begin{bmatrix} a_{11} & a_{12} \\ a_{21} & a_{22} \end{bmatrix}^{-1}$ has $b_{22}=0$, and similarly if $a_{22}=0$, since 
\[
\begin{bmatrix}0 & a_{12} \\ a_{21} & a_{22}\end{bmatrix}
\begin{bmatrix}-a_{22}a_{12}^{-1}a_{21}^{-1} & a_{21}^{-1} \\ a_{12}^{-1} & 0\end{bmatrix}= 
\begin{bmatrix}1 & 0 \\ 0 & 1\end{bmatrix}.
\]
A genuine $2\times 2$ edge elimination is needed only when both $a_{11}$ and $a_{22}$ are $0$.
If $a_{22}\neq 0$, one may instead first eliminate vertex $2$ (equivalently, after swapping the first two vertices) and then eliminate vertex $1$ in the resulting Schur complement.
When both $a_{11}$ and $a_{22}$ are $0$, edge elimination defines two successive columns of $L$, namely $\begin{bmatrix} a_{:,2}a_{12}^{-1} & a_{:,1}a_{21}^{-1}\end{bmatrix}$, and an anti-diagonal block in $D$, namely $\begin{bmatrix} 0 & a_{12} \\ a_{21} & 0 \end{bmatrix}$.

A sequence of vertex and edge eliminations on consecutive vertices results in an LDL factorization after $\rank(A)$ vertices are eliminated as part of either type of elimination.
With row and column permutations $P$ chosen to ensure the eliminations act on consecutive rows and columns, we obtain
\[P^TAP = LDL^H,\]
where $D$ has $1\times 1$ and $2\times 2$ diagonal blocks.
Each vertex elimination contributes one column of $L$ and one $1\times 1$ block of $D$, while each genuine edge elimination contributes two columns of $L$ and one anti-diagonal $2\times 2$ block.

When $\mathbb{F}=\mathbb{R}$ or $\mathbb{F}=\mathbb{C}$, to minimize round-off error in finite precision, row and column permutations in LDL are typically selected based on adaptations of partial pivoting.
In particular, the Bunch-Kaufman pivoting algorithm~\cite{bunch1977some} and its variants~\cite{dongarra1990solving,higham1997stability,duff1979direct,duff1991factorization}, seek the largest magnitude off-diagonal element in the leading column ($a_{i1}$), and choose between performing an edge elimination corresponding to that off-diagonal entry, or a vertex elimination of one of the vertices adjacent to this edge.
The choice is made based on the magnitudes of $a_{i1}$, $a_{11}$, and $a_{ii}$.
This algorithm leads to a bounded growth factor (growth in magnitude of $L$ entries with the number of vertex/edge eliminations), and hence favorable numerical stability bounds~\cite{higham1997stability}.
%of off-diagonal entry selected and the two diagonal

\subsection{Saddle-Point Systems of Linear Equations}
In this subsection, $A$ denotes the leading block of a saddle-point system and $B$ the constraint block.
Linear systems of equations with saddle-point structure are ubiquitous in numerical optimization methods as they reflect the structure of the Karush-Kuhn-Tucker (KKT) systems that describe minima of convex quadratic optimization problems~\cite{nocedal1999numerical,dyn1983numerical,gould2001solution}.
These systems are described by a matrix with the following block structure,
\[M=\begin{bmatrix} A & B^H \\ B & 0\end{bmatrix},\]
where $A$ is $H$-symmetric.
In the context of optimization, $A$ often corresponds to a Hessian, while $B$ encodes constraints.
In particular, when $A$ is real and symmetric positive definite (SPD), 
\[M\begin{bmatrix} x \\ \lambda \end{bmatrix} = \begin{bmatrix} y \\ z \end{bmatrix}\]
describes the minimum $x$ in the optimization problem,
\[\min_{x,Bx=z} \frac 12x^TAx - x^Ty.\]
Numerical solutions to systems with saddle point structure have been studied for decades~\cite{benzi2005numerical}.

One commonly studied approach~\cite{golub2003solving,benzi2005numerical,karim2022efficient} to direct factorization of $M$ is to leverage positive definiteness of $A$ (when the Hessian is semi-definite, Tikhonov regularization is an intuitive and popular option to restore SPD structure~\cite{dollar2007using,benzi2005numerical}).
Then, the Schur complement of $A$, namely $-BA^{-1}B^H$, is negative semi-definite.
Further, in the context of optimization, $B$ is typically full row rank, so Cholesky may be applied to $BA^{-1}B^H$ to obtain an LDL factorization of $M$, in which the $D$ factor is diagonal.
Use of (pivoted) Cholesky instead of a general pivoted LDL factorization is advantageous in the sparse case, in which pivoting has the additional objective of reducing fill.

However, eliminating $A$ first may be unfavorable in terms of fill.
We illustrate this potential overhead with the following example.
%As an example, consider a saddle point system with the following nonzero structure.
\begin{example}\label{ex:1}
Consider a saddle point system $M$ with blocks,
\[A = \begin{bmatrix} a_{11} & a_{12} & & 0 \\
a_{21} & 0 & \ddots  &  \\
 & \ddots &  \ddots & a_{n-1,n} \\
 0& & a_{n,n-1} & 0\\
\end{bmatrix}, \quad 
B = \begin{bmatrix} b_{11}  & & 0 \\
\vdots & \ddots & \\
b_{n1} & 0& b_{nn}
\end{bmatrix}.
\]
Elimination of the first column in $M$ with these blocks has cost $O(n^2)$ and results in dense Schur complement block that is later zeroed out.
Eliminating the first $n$ columns in $M$ in order hence leads to $O(n^3)$ cost and $O(n^2)$ nonzeros in $L$.
On the other hand, performing edge eliminations of pairs of consecutive columns of $A$ in reverse order, results in $O(n^2)$ cost and $O(n)$ nonzeros in the $L$ factor.
Further, if we instead perform edge eliminations of $b_{ii}$ in either order in $i$, there is no fill in $L$ and the overall cost of LDL factorization of $M$ is $O(n)$.
\end{example}
The last approach in Example~\ref{ex:1} corresponds to the null-space method in constrained optimization, which we discuss next.

\subsubsection{Null-space Method}
\label{sec:bg_np}

Given a saddle point system of the form,
\[M = \begin{bmatrix} A & B^H \\ B & 0 \end{bmatrix},\]
where $A$ is $n\times n$ and $B$ is $m\times n$ with full row rank $m\le n$, the null-space method solves $Mx=b$ by projecting $A$ onto the null space of $B$~\cite{heath2018scientific,antoniou2007practical,benzi2005numerical}.
In the classical optimization setting, one additionally assumes that $A$ is SPD on $\ker(B)$.
A classical way to compute the null-space is via the QR factorization of $B$.
In 2009, Schilders proposed a method based on the LQ factorization of $B$~\cite{schilders2009solution}, leveraging this factorization to construct a mixed orthogonal/triangular factorization of $M$.

In more recent works, the null-space method has been interpreted in terms of generalized block triangular matrix factorizations~\cite{rees2014null} and more specifically in terms of a Schur complement~\cite{pestana2016null}.
Assuming, after a column permutation, that $B=\begin{bmatrix} B_1 & B_2\end{bmatrix}$ where $B_1$ is $m\times m$ and nonsingular, we may permute $M$ as
\[P^TMP=P^T
\begin{bmatrix} A_{11} &A_{12}& B_{1}^H   \\A_{21} &A_{22} & B_{2}^H \\ B_{1} & B_{2} & 0\end{bmatrix}
P 
=
\begin{bmatrix} A_{11} & B_{1}^H &A_{12}  \\ B_{1}  & 0& B_{2}\\A_{21}  & B_{2}^H&A_{22}\end{bmatrix}.
\]
Then the Schur complement of the leading $2m\times 2m$ block in $P^TMP$ is the null-space matrix~\cite{pestana2016null}.

Further, as pointed out in~\cite{rees2014null,schilders2009solution,gould1999modified}, this $2m\times 2m$ block may be factorized in a structure-preserving way by reordering into a 2-by-2 tiling with the following nonzero structure,
\[W = P_2^T 
\begin{bmatrix} A_{11} & B_{1}^H  \\ B_{1}  & 0
\end{bmatrix}P_2 = 
\begin{bmatrix}
w_{11} & w_{12} & w_{13} & w_{14} &\cdots \\
w_{21} & 0 & w_{23} & 0 &\cdots \\
w_{31} & w_{32} & w_{33} & w_{34} &\cdots \\
w_{41} & 0 & w_{43} & 0 &\cdots \\
\vdots & \vdots & \vdots & \vdots & \ddots
\end{bmatrix}.
\]
Edge eliminations on pairs of vertices in this ordering preserve the sparsity of the 2-by-2 blocks.
We refer to the above tiling scheme as $\Gamma$-tiling.
This $2\times 2$ tiling underlies the saddle-point elimination procedures developed in Section~\ref{sec:nsldl}. Algorithm~\ref{alg:ldl-gel} implements the $\Gamma$-tiling algorithm with handling of zero $2\times 2$ diagonal blocks, whereas Algorithm~\ref{alg:ldl-nsp} realizes the same step via an LU factorization of $B^H$.
%, which is
%\[
%A_{22} - 
%\begin{bmatrix}
%A_{21}  & B_{2}^H
%\end{bmatrix}
%\begin{bmatrix} 0 & B_{1}^{-1} \\ B_{1}^{-H} & B_{1}^{-H}A_{11}B_{1}^{-1}\end{bmatrix}
%\begin{bmatrix}
%A_{21}^H  \\ B_{2}
%\end{bmatrix}
%= A_{22} - A_{21}B_{1}^{-1}B_2 - B_2^HB_{1}^{-H}A_{21}^H - B_2^HB^{-H}_{11}A_{11}B_1^{-1}B_2
%\]
%
%
%

%Then the Schur complement of the leading $2R\times 2R$ block is full rank.
%The null-space method corresponds to taking the Schur complement of this block within the above system~\cite{pestana2016null}, which is
%\[
%P^HMP-
%\begin{bmatrix} A_{11} & B_{11}^H  \\ B_{11}  & 0\\A_{21}  & B_{12}^H\\ B_{21} & 0 \end{bmatrix}
%\begin{bmatrix} 0 & B_{11}^{-1} \\ B_{11}^{-H} & B_{11}^{-H}A_{11}B_{11}^{-1}\end{bmatrix}
%\begin{bmatrix} A_{11} & B_{11}^H  \\ B_{11}  & 0\\A_{21}  & B_{12}^H\\ B_{21} & 0 \end{bmatrix}^H
%=\begin{bmatrix} 0&0&0&0\\0&0&0&0\\0&0& X &0 \\ 0&0&0&0\end{bmatrix}.
%\]

\subsection{$\Gamma$-Elimination}
\label{subsec:game}
%Given a saddle point system of the form,
%\[M = \begin{bmatrix} A & B^H \\ B & 0 \end{bmatrix},\]
%where $A\in\mathbb{F}^{n\times n}$ is symmetric and $B\in\mathbb{F}^{m\times n}$ is full rank,
Following the $\Gamma$-tiling scheme introduced at the end of the prior section, we describe an edge-elimination process for computing an LDL factorization while preserving saddle-point structure.
We refer to this method, which was studied by Gould~\cite{gould1999modified}, as $\Gamma$-elimination.
After any stated row/column permutation, we continue to denote the permuted blocks by $A$, $B$, and $M$.
Graphically, each edge elimination below is just the Schur-complement update associated with a $2\times 2$ pivot having one endpoint in the $A$-block and one in the $B$-block.
Given the graph $G$ corresponding to the off-diagonal part of $M$, let the vertices corresponding to the first $n$ rows/columns of $M$ be $V_1$ and the remaining vertices be $V_2$.
There exists a sequence of $\rank(B)$ edge eliminations on edges $(u,v)$ where $u\in V_1$ and $v\in V_2$.
The sequence may be obtained by repeatedly picking an edge corresponding to a nonzero entry in $B$, performing the corresponding edge elimination, and recursing.
The key invariant is that the Schur complement maintains saddle-point structure (the lower right block stays zero).

To obtain this sequence, pick a nonzero entry $b_{11}$ in $B$ after permuting rows and columns so that it lies in the first row and first column of $B$.
Then, after the corresponding edge elimination, the updated system $M'$ is
\begin{align*}
M'
&=
M- 
\begin{bmatrix} 
a_{:,1} & b_{1,:}^H \\ 
b_{:,1} & 0 
\end{bmatrix}
\begin{bmatrix} a_{11} & b_{11}^* \\ b_{11} & 0 \end{bmatrix}^{-1}
\begin{bmatrix} 
a_{:,1} & b_{1,:}^H \\ 
b_{:,1} & 0 
\end{bmatrix}^H\\
&=
\begin{bmatrix} A' & {B'}^H \\ B' & 0\end{bmatrix},
\end{align*}
where
\begin{align*}
A'&=A- a_{:,1}b_{11}^{-1}b_{1,:}-b_{1,:}^H(b_{11}^*)^{-1}a_{:,1}^H+b_{1,:}^H a_{11}(b_{11}^*b_{11})^{-1}b_{1,:}, \\
B'&=B - b_{:,1}b_{11}^{-1}b_{1,:}.
\end{align*}
Hence, this sequence of eliminations simultaneously performs Gaussian elimination (LU factorization) on $B$, which will result in the off-diagonal block in the saddle point system being zero after $\rank(B)$ steps.
It is important to note that these edge eliminations connect pairs of vertices in $V_1$ only if one of them is connected to a vertex in $V_2$ (equivalently, a new nonzero in $A'$ may only be introduced through a shared neighbor in $V_2$), while not connecting any nodes in $V_1$ to $V_2$ that were previously disconnected (since columns of $B'$ are linear combinations of those in $B$).
The first row and column of $A'$ and $B'$ are zero and may be dropped from the system, while the remaining rows and columns are a saddle-point system.

After $\rank(B)$ edge eliminations, we obtain a partial LDL factorization (or a complete one if $\rank(M)=2\rank(B)$).
We could minimize the number of degrees of freedom in representing this LDL, by storing a rescaled version of the rank-2 skeleton decomposition~\cite{kishore2017literature},
\begin{align*}
\begin{bmatrix} 
a_{:,1} & b_{1,:}^H \\ 
b_{:,1} & 0 
\end{bmatrix}&
\begin{bmatrix}0  & b_{11}^{-1} \\ (b_{11}^*)^{-1} & -a_{11}(b_{11}^*b_{11})^{-1} \end{bmatrix}
\begin{bmatrix} 
a_{:,1} & b_{1,:}^H \\ 
b_{:,1} & 0 
\end{bmatrix}^H \\&=
\underbrace{\begin{bmatrix} 
a_{11} & 1  \\ 
a_{2:,1}& b_{1,2:}^H/b_{11}^* \\
b_{:,1} & 0 
\end{bmatrix}}_{K}
\underbrace{
\begin{bmatrix}0  & 1 \\ 1 & -a_{11} \end{bmatrix}}_{D_{2\times 2}}
\begin{bmatrix} 
a_{11} & 1  \\ 
a_{2:,1}& b_{1,2:}^H/b_{11}^* \\
b_{:,1} & 0 
\end{bmatrix}^H.
\end{align*}
To recover the standard LDL format from the above rank-2 decomposition, we may shift the $(n+1)$th row to be second via permutation matrix $P'$ and define two columns of $L$ (a lower trapezoidal and unit diagonal $n\times 2$ matrix) and the corresponding 2-by-2 diagonal or antidiagonal part of the $D$ matrix as follows,
\begin{enumerate}
\item if $a_{11}=0$, they are $P'K\begin{bmatrix} 0 & 1 \\ 1 & 0\end{bmatrix}$ and $\begin{bmatrix} 0 & 1 \\ 1 & 0\end{bmatrix}$,
\item if $a_{11}\neq 0$,
they are $P'KK_1^{-1}L_D$ and $\tilde{D}$, where $K_1 = \begin{bmatrix} a_{11} & 1 \\ b_{11} & 0 \end{bmatrix}$, and we use the $2\times 2$ LDL factorization, $K_1D_{2\times 2}K_1^H=L_D\tilde{D}L_D^H$.
%they are $PKL_D^{-H}\tilde{D}^{-1}$ and $\tilde{D}$, where
%$L_D=\begin{bmatrix} 1 & 0 \\ 1/a_{11} & 1 \end{bmatrix}$ and $\tilde{D}= \begin{bmatrix} 1/a_{11} & 0 \\ 0 &
%-a_{11}^2\end{bmatrix}$.
\end{enumerate}
%standard form may be recovered by independent factorizations of triangular matrices $\begin{bmatrix} 0 & D_B[i,i] \\ D_B[i,i]^* & D_{\tilde{A}}[i,i] \end{bmatrix}^{-1}=L_iL_i^H$ and absorbing $L_i^{-H}$ into the full $L$ factor (to obtain columns $2i-1$ and $2i$).

%\section{Preserving Saddle Point Structure in LDL Factorization}
\section{Adapting the Null-Space Method to LDL Factorization}
\label{sec:nsldl}

\begin{algorithm}[h]
\caption{$[\B P,\B Q,\B Y,\B L,\B U,\B D,r]=$ Partial-LDL-$\Gamma$-Elim($\B A$, $\B B$)}\label{alg:ldl-gel}
\begin{algorithmic}[1]
\Require $M=\begin{bmatrix} A & B^H \\ B & 0 \end{bmatrix}$, 
$\B A\in\mathbb{F}^{n\times n}$, 
$\B A^H = \B A$,
$\B B\in\mathbb{F}^{m\times n}$. 
\Ensure The rank of $B$ is $r$ and $P,Q,Y,L,U,D$ give a partial LDL factorization of $M$ satisfying Definition~\ref{def:partLDL}.
\State Pick any nonzero $b_{ij}$ in $B$ (if none, return null matrices and $r=0$).
\State Let $P_1=I_n-(e_i-e_1)(e_i-e_1)^T$
\State Let $Q_1=I_m-(e_j-e_1)(e_j-e_1)^T$
\State Let $\begin{bmatrix} \bar{A} & \bar{B}^H \\ \bar{B} & 0\end{bmatrix}=\begin{bmatrix} P_1 & 0 \\ 0 & Q_1\end{bmatrix}^TM\begin{bmatrix} P_1 & 0 \\ 0 & Q_1\end{bmatrix}$
\State $d_{11} = -\bar{a}_{11}$
\State $y_1 = \bar{a}_{:,1} - e_1\bar{a}_{11}$
\State $l_1 = \bar{b}_{1,:}^H/\bar{b}_{11}^*$
\State $u_1 = \bar{b}_{:,1}^H$
\State $A_{\text{new}} = \bar{A}- \bar{a}_{:,1}\bar{b}_{11}^{-1}\bar{b}_{1,:}-\bar{b}_{1,:}^H(\bar{b}_{11}^*)^{-1}\bar{a}_{:,1}^H+\bar{b}_{1,:}^H\bar{a}_{11}(\bar{b}_{11}^*\bar{b}_{11})^{-1}\bar{b}_{1,:}$
\State $B_{\text{new}} = \bar{B} - \bar{b}_{:,1}\bar{b}_{11}^{-1}\bar{b}_{1,:}$
\State $[\B P_2,\B Q_2,\B Y_2,\B L_2,\B U_2,\B D_2,r_2]=$ Partial-LDL-$\Gamma$-Elim($A_{\text{new}}[2:,2:]$, $B_{\text{new}}[2:,2:]$)
\State $\tilde{Y}_2 = \begin{bmatrix} 0 \\ Y_2\end{bmatrix}$, $\tilde{L}_2 = \begin{bmatrix} 0 \\ L_2\end{bmatrix}$, $\tilde{U}_2 = \begin{bmatrix} 0 & U_2\end{bmatrix}$
\State $Y = \begin{bmatrix} y_1 & \tilde{Y}_2 \end{bmatrix}$, $L = \begin{bmatrix} l_1 & \tilde{L}_2 \end{bmatrix}$, $U = \begin{bmatrix} u_1 \\ \tilde{U}_2 \end{bmatrix}$
\State $D = \begin{bmatrix} d_{11} & 0 \\ 0 & D_{2}\end{bmatrix}$
\State $r=r_2+1$
\State $\tilde{P}_2 = \diag(1,P_2)$ and $\tilde{Q}_2 = \diag(1,Q_2)$
\State $P=P_1\tilde{P}_2$
\State $Q=Q_1\tilde{Q}_2$
\end{algorithmic}
\end{algorithm}
Algorithm~\ref{alg:ldl-gel} makes the elimination step explicit and motivates Definition~\ref{def:partLDL}. Algorithm~\ref{alg:ldl-nsp} below derives the same partial factorization in a form better suited to the later dense and sparse algorithms and to the corresponding cost analysis.

We now adapt the idea of Schilders' null-space method~\cite{schilders2009solution} to devise efficient algorithms for LDL factorization over a general field.
Over GF(2), matrices (vector spaces) do not admit orthogonal factorizations (bases), and a suitable $\Gamma$-tiling ($\Gamma$-elimination order) is hard to identify statically (zeros may be frequently introduced by updates).
%We first give a basic adaptation of $\Gamma$-tiling that finds an LDL factorization of a saddle-point system via edge eliminations.
We present a modified version of Schilders' algorithm, which leverages LU with row and column permutations to identify a null-space.
The algorithm allows for a simple reduction of the null-space LDL factorization for saddle-point systems to rectangular LU factorization and LDL of a dense matrix.

\subsection{Constraint-Complemented Partial LDL}

In the $\Gamma$-elimination process, 
we may keep the rows and columns in their original blocks to express the resulting partial LDL factorization obtained as the following partial block factorization.
In this decomposition, we leverage the structure of the rescaled rank-2 skeleton decompositions obtained from each edge elimination.
The resulting decomposition is a block representation of the Schur complement factorization described in~\cite{rees2014null} (Section 2.2).
\begin{definition}[Constraint-Complemented Partial LDL]
\label{def:partLDL}
Consider a saddle point system,
\[M = \begin{bmatrix} A & B^H \\ B & 0 \end{bmatrix},\]
where $A\in\mathbb{F}^{n\times n}$ is $H$-symmetric and $B\in\mathbb{F}^{m\times n}$ is rank $r$.
%$D_{\tilde{A}}$,
The partial LDL decomposition of $M$ is
defined by matrices 
$D\in\mathbb{F}^{r\times r}, L, Y\in\mathbb{F}^{n\times r}, U\in\mathbb{F}^{r\times m},P\in\{0,1\}^{n \times n}, Q\in\{0,1\}^{m\times m}$,
where $D$ is diagonal and $D^H=D$, $L$ is lower trapezoidal and unit-diagonal, $Y=\begin{bmatrix}Y_1 \\ Y_2 \end{bmatrix}$ is lower trapezoidal with zero diagonal and $Y_1$ is square, $U$ is upper trapezoidal with nonzero diagonal, while $P$ and $Q$ are permutation matrices.
These matrices give a valid constraint-complemented partial LDL of $M$ if, for $V=\begin{bmatrix}Y_1 - D \\ Y_2\end{bmatrix}$, we have
%symmetric diagonal 
%and $L_B\in\mathbb{F}^{n\times m}$
%with $L_B$ additionally unit diagonal, upper triangular matrix $U_B\in\mathbb{F}^{n\times m}$, and permutation matrices $P_A$ and $P_B$, so
%There e some diagonal matrix $D$, 
%and upper triangular matrix $U_B$, we have
\begin{align}
\begin{bmatrix} P & 0\\ 0& Q \end{bmatrix}^TM\begin{bmatrix} P & 0\\ 0& Q \end{bmatrix} 
&=
\begin{bmatrix} V & L \\ U^H & 0\end{bmatrix}
\begin{bmatrix} 0 & \B I \\ \B I &\B D\end{bmatrix}
\begin{bmatrix} V & L \\ U^H & 0\end{bmatrix}^H
\\
&+ \begin{bmatrix} P^TAP -VL^H - LV^H-LDL^H&0\\0&0\end{bmatrix},
\label{eq:part_blk_ldl}
\end{align}
where $P^TAP -VL^H - LV^H-LDL^H$ is nonzero only in the lower right $(n-r)\times(n-r)$ block.
%and $\forall i\in\{1,\ldots, m\}$,
%$\begin{bmatrix} y_{ii} & 1 \\ 1 & 0 \end{bmatrix}^{-1} = \begin{bmatrix} 0 & 1 \\ 1 & d_{ii} \end{bmatrix}$.
\end{definition}
%In the above, edge eliminations that are performed as two vertex eliminations (one diagonal element in the $2\times 2$ leading block in the permuted matrix is nonzero) are left in unfactorized form.
%In the case of a partial factorization, w
Algorithm~\ref{alg:ldl-gel} constructs the above partial decomposition via $\Gamma$-elimination.
If $A$ is skew-symmetric rather than $H$-symmetric (and $A$ has zero diagonal), we have $D=0$, while if $A$ is positive definite, $\diag(D)$ is negative.
To obtain a full LDL decomposition of $M$ from a partial one, it suffices to compute and append the LDL factorization of the Schur complement, the $(n-r)\times (n-r)$ nonzero sub-block of $P^TAP - VL^H - LV^H-LDL^H$.
%%A complete LDL factorization of $M$ may be obtained from the above by additionally
%%\begin{itemize}
%%\item Interleaving the blocks to obtain the sequence of updates from vertex and edge eliminations,
%%\[\]terms from the above block factorization, i.e.,
%%\item performing the LDL factorization of 
%%\end{itemize}

The block factorization above simultaneously performs Gaussian elimination on $B$; in particular, it yields an LU factorization of $P^TB^HQ = LU$.
The next subsection answers the converse algorithmic question.
Namely, we give an algorithm to obtain a partial LDL factorization of a saddle point system from the LU factorization of $B^H$.
%corresponding to a valid $\Gamma$-tiling factorization order, by first doing a pivoted LU factorization of $B^H$.

\subsection{LDL Factorization via an LU-based Schilders Null-space Method}
\begin{algorithm}[h]
\caption{$[\B P, Q, \B Y,L,U,\B D,r]=$ Null-space-partial-LDL-Schilders($\B A$, $B$)}\label{alg:ldl-nsp}
\begin{algorithmic}[1]
\Require $M=\begin{bmatrix} A & B^H \\ B & 0 \end{bmatrix}$, 
$\B A\in\mathbb{F}^{n\times n}$, 
$\B A^H = \B A$,
$\B B\in\mathbb{F}^{m\times n}$. 
\Ensure $r$ is the rank of $B$ and $P,Q,Y,L,U,D$ give a partial LDL factorization of $M$ satisfying Definition~\ref{def:partLDL}.
%\begin{align*}
%\begin{bmatrix} P & 0\\ 0& Q \end{bmatrix}^TM\begin{bmatrix} P & 0\\ 0& Q \end{bmatrix} 
%&=
%\begin{bmatrix} Y & U^H \\ L & 0\end{bmatrix}
%\begin{bmatrix} 0 & \B D_B \\ \B D_B^H &\B D\end{bmatrix}
%\begin{bmatrix} Y & U^H \\ L & 0\end{bmatrix}^H
%\\
%&+ \begin{bmatrix} A -Y D_B^HL^H - LD_BY^H-LDL_{B}^H&0\\0&0\end{bmatrix}.
%\label{eq:part_blk_ldl}
%\end{align*}
%where $P$ and $Q$ are permutation matrices, $Y\in\mathbb{F}^{n\times R_B}$ and $L\in\mathbb{F}^{m\times R_B}$ are lower trapezoidal and nonzero on the diagonal, $U\in\mathbb{R}^{n\times R_B}$ is upper triangular and has nonzero diagonal, while $D_B$ and $D$ are diagonal and symmetric. Additionally, $D_B$ is invertible and $\forall i\in\{1,\ldots, R_B\}$, \[\begin{bmatrix} Y[i,i] & L[i,i]^* \\ L[i,i] & 0 \end{bmatrix}^{-1} = \begin{bmatrix} 0 & D_B[i,i] \\ D_B[i,i]^* & D[i,i]\end{bmatrix}.\]
%%$\B P^T\B M \B P = \B L \B D \B L^T$, $\B L$ is unit-diagonal and lower-triangular, $\B D$ is block diagonal with blocks of unit size or $2\times 2$ and antidiagonal, and $\B D^H = D$
%of the form $\begin{bmatrix} 0 & 1 \\ 1 & 0\end{bmatrix}$, the last $n-R$ rows and columns of $\B D$ are zero, $R$ is the rank of $\B A$
\State Compute an LU factorization of the form $P^TB^HQ = LU$, where $L\in\mathbb{F}^{n\times r}$ is lower trapezoidal and unit diagonal, $U\in\mathbb{F}^{r \times m}$ is upper trapezoidal, and $r$ is the rank of $B$.
\State Let $L=\begin{bmatrix}L_1 \\ L_2\end{bmatrix}$, $L_1\in\mathbb{F}^{r\times r}$. 
%\State Form 
\State Consider the block system,
\[ 
\begin{bmatrix} P & 0 \\ 0 & Q\end{bmatrix}^TM\begin{bmatrix} P & 0 \\ 0 & 

Q\end{bmatrix}=
%P^T
\begin{bmatrix} A_{11} &A_{12}& B_{1}^H   \\A_{21} &A_{22} & B_{2}^H \\ B_{1} & B_{2} & 0\end{bmatrix},
%P 
%=
%\begin{bmatrix} A_{11} & B_{1}^H &A_{12}  \\ B_{1}  & 0& B_{2}\\A_{21}  & B_{2}^H&A_{22}\end{bmatrix},
\]
where $A_{11}\in\mathbb{F}^{r\times r}$.
\State $W=L_1^{-1}A_{11}L_1^{-H}$
\State $D = -\diag(W)$
\State $Y_1 = L_1\text{tril}(W) + D$ \quad \Comment{\text{tril}($W$) extracts the lower triangular part of $W$, inclusive of the diagonal}
\State $Y_2 =
\Big(A_{21} - L_2((Y_1^H-D) 
+  DL_1^H)
\Big)L_1^{-H}$
\State $Y = \begin{bmatrix}Y_1 \\ Y_2\end{bmatrix}$
\end{algorithmic}
\end{algorithm}
Algorithm~\ref{alg:ldl-nsp} provides a way to construct a constraint-complemented partial LDL of $M$ from a rank-reduced LU factorization of $B^H$.
The key ingredient is the block identity below, which is the algebraic step underlying Schilders' factorization~\cite{schilders2009solution}.
Namely, if $n=r$ in the constraint-complemented partial LDL factorization, we have
\begin{align*}
P^TAP &= VL^H + LV^H + LDL^H,  \\
W&=L^{-1}P^TAPL^{-H} = \underbrace{L^{-1}V}_{\text{lower triangular}} + \underbrace{V^HL^{-H}}_{\text{upper triangular}}+\underbrace{D}_{\text{diagonal}}.
\end{align*}
It follows that $D$ is the negation of the diagonal part of $W$, while $Y_1$ may be obtained from the lower triangular part of $W$.
The formula for $Y_2$ then follows by equating the $(2,1)$ block in $P^TAP = VL^H + LV^H + LDL^H$.
\begin{proposition}
\label{prop:alg31_alg32_equiv}
Let $M=\begin{bmatrix} A & B^H \\ B & 0 \end{bmatrix}$ and let $P^TB^HQ=LU$ be the rank-revealing LU factorization used in Algorithm~\ref{alg:ldl-nsp}. The matrices $P,Q,Y,L,U,D$ constructed by Algorithm~\ref{alg:ldl-nsp} satisfy Definition~\ref{def:partLDL}. Moreover, they realize the same constraint-complemented partial LDL factorization obtained by carrying out the corresponding $\Gamma$-elimination sequence of Algorithm~\ref{alg:ldl-gel}.
\end{proposition}
\begin{proof}
Algorithm~\ref{alg:ldl-gel} performs a sequence of edge eliminations on nonzero entries of $B$, and therefore simultaneously carries out Gaussian elimination on $B$. For the same sequence of row and column pivots, the resulting lower- and upper-trapezoidal factors are exactly the $L$ and $U$ in the LU factorization $P^TB^HQ=LU$.
Once $P$, $Q$, $L$, and $U$ are fixed, the defining identity~\eqref{eq:part_blk_ldl} determines $D$ and $Y$ blockwise. Partition $P^TAP$ conformally with $L=\begin{bmatrix} L_1 \\ L_2\end{bmatrix}$ and write $V=\begin{bmatrix} Y_1-D \\ Y_2\end{bmatrix}$. The $(1,1)$ block of~\eqref{eq:part_blk_ldl} gives
\[
W=L_1^{-1}A_{11}L_1^{-H}=L_1^{-1}(Y_1-D)+(Y_1^H-D)L_1^{-H}+D,
\]
so $D=-\diag(W)$ and $Y_1=L_1\operatorname{tril}(W)+D$. The $(2,1)$ block gives
\[
A_{21}=Y_2L_1^H+L_2(Y_1^H-D)+L_2DL_1^H,
\]
which is exactly the formula used for $Y_2$ in Algorithm~\ref{alg:ldl-nsp}. Thus Algorithm~\ref{alg:ldl-nsp} reconstructs the same partial factorization represented explicitly by Algorithm~\ref{alg:ldl-gel}.
\end{proof}
Algorithm~\ref{alg:ldl-nsp} may be implemented with an LU factorization, a triangular inverse, and a few matrix-matrix products.
%s, triangular inversion, and pivoted LU factorization.
All three operations may be done in matrix-multiplication time~\cite{bunch1974triangular,IBARRA198245}.
Further, we expect that the algorithm satisfies reasonable stability bounds provided each component is performed stably. We note that triangular inversion is not necessarily unstable~\cite{croz1992stability}. Use of triangular solves (backward and forward substitution) instead of inversion could improve stability, and may also be done in matrix multiplication time when the number of right hand sides is the same as the number of equations.
%{ Inversion of diagonal blocks would allow for a trade-off between pure triangular inversion and solves~\cite{wicky2017communication}.}.
We leave a more complete stability analysis and numerical evaluation for future work.
Obtaining a full LDL of a saddle-point system still requires a dense LDL of the Schur complement (in addition to transforming to standard LDL format, which may be done cheaply in the same way as described at the end of Section~\ref{subsec:game}).
In the next section, we provide fast-matrix-multiplication-based algorithms for dense factorization that effectively handle low-rank scenarios.

\section{Dense Matrix Factorization in Matrix Multiplication Time}
\label{sec:app_ldl_strassen}
In this section, we discuss linear algebra subroutines for pivoted dense LU and LDL factorization of matrices over a general field, providing a new algorithm for the latter.
In this section, $A$ denotes a general dense matrix (rectangular in the LU subsection and square $H$-symmetric in the LDL subsection), rather than the leading block of a saddle-point system.
The dense recursion uses the same saddle-point reduction as Section~\ref{sec:nsldl}: when the leading recursive block is too low rank, the remaining work is reorganized through an off-diagonal block and reduced to LU.
%These results allow us to state bounds on Clifford circuit simulation time in terms of the matrix multiplication complexity exponent, $\omega$, reproducing prior work on graph state simulation~\cite{gosset2024fast}.
%The use of fast matrix multiplication in Gaussian elimination was first proposed alongside the first subcubic matrix multiplication algorithm in the seminal work by Strassen~\cite{strassen1969gaussian} in 1969.
%Five years later, Bunch and Hopcroft~\cite{bunch1974triangular} provided an approach to perform LU factorization with partial pivoting in $O(n^\omega)$ time for a full-rank $n\times n$ matrix.
For LU, we could adapt the method of Ibarra, Moran, and Hui~\cite{IBARRA198245} with a minor post-processing cost, via the method of Jeannerod~\cite{jeannerod2006lsp}.
We give a simple recursive direct LU algorithm which performs row and column pivoting and achieves the cost $O(mnr^{\omega-2})$ for $m\times n$ matrices of rank $r$.
Given that the same complexity may be achieved with existing methods~\cite{jeannerod2006lsp}, we defer the algorithm and its analysis to Appendix~\ref{app:lu}, but leverage the following result from therein\footnote{The sparse algorithms below require only the rank-revealing factorization and associated permutations; the row-echelon property is retained because it gives a convenient description of the rectangular LU output.}.
\begin{theorem}
\label{thm:LU_dense_fast}
\sloppy
Given $\B A \in \mathbb{F}^{m\times n}$, Algorithm~\ref{alg:lup} computes the factorization $\B P^T\B A\B Q = \B L \B U$ and identifies the rank $r$ of $\B A$ using
\[T_\text{LU}(m,n,r)=O(r^{\omega-2}mn)=O(\min(m,n)^{\omega-1}\max(m,n))\]
arithmetic operations.
Further, $(P^TL)^H$ is in row-echelon form.
\end{theorem}
In the above theorem, by row-echelon form, we mean any matrix $T$ that may be written in the form,
%so that $T^T$ is in a row-echelon-like form, i.e., $T$ can be written in the form,
\[
T=\begin{bmatrix}t_1 & \cdots & t_n\end{bmatrix}^H,
\qquad
t_i=\begin{bmatrix} 0 \\ \vdots \\ 0 \\ r_i \end{bmatrix},
\]
where the first entry of each $r_i$ is nonzero and
\[
\forall i\in\{2,\ldots,n\},\qquad \dim(r_i)<\dim(r_{i-1}).
\]

For LDL factorization, we must consider symmetric pivoting, and rank-deficient (sub)matrices.
As part of our LDL algorithm, we leverage the fast LU factorization of a rectangular matrix.
%We provide simple algorithms and corresponding $\omega$-dependent cost bounds below for completeness.
%We first consider dense matrices, then sparse matrices corresponding to graphs with a given treewidth decomposition.

\subsection{LDL Factorization}

\begin{algorithm}[!hp]
\caption{$[\B P,\B L,\B D,r]=$ Fast-LDL($\B A$)}\label{alg:ldl}
\begin{algorithmic}[1]
\Require $\B A\in\mathbb{F}^{n\times n}$, $\B A^H = \B A$
\Ensure $\B P^T\B A \B P = \B L \B D \B L^H$, $\B L$ is unit-diagonal and lower-triangular, $\B D$ is block diagonal with blocks of unit size or of the form $\begin{bmatrix} 0 & 1 \\ 1 & 0\end{bmatrix}$, the last $n-r$ rows and columns of $\B D$ are zero, $r$ is the rank of $\B A$
\If {$n\leq 3$}
 \State Obtain $\B P, \B L, \B D, r$ by direct factorization of $\B A$.
\Else 
\State $\B A = \begin{bmatrix} \B A_{11} & \B A_{12} \\ \B A_{12}^H & \B A_{22}\end{bmatrix}$, $\B A_{22}\in\mathbb{F}^{\lfloor n/3 \rfloor\times \lfloor n/3\rfloor}$
\State $[\B P_1, \B L_{11}, \B D_1, r_1] = \text{Fast-LDL}(\B A_{11})$
%,$A_{11}\in\mathbb{F}^{\lfloor 2n/3 \rfloor \times \lfloor 2n/3 \rfloor}$
\State $\tilde{\B P_1} = \begin{bmatrix} \B P_1[:,1:r_1]  & 0  & \B P_1[:,r_1+1:]\\ 0 & \B I & 0 \end{bmatrix}$ so $\tilde{\B P_1}\in\mathbb{F}^{n\times n}$
\State $\tilde{\B A} = \tilde{\B P}_1^T \B A \tilde{\B P}_1= \begin{bmatrix} \tilde{\B A}_{11} & \tilde{\B A}_{12} \\ \tilde{\B A}_{12}^H & \tilde{\B A}_{22}\end{bmatrix}$, $\tilde{\B A}_{11}\in\mathbb{F}^{r_1\times r_1}$
\State $\tilde{L}_{11}=L_{11}[1:r_1,1:r_1]$, $\tilde{D}_1=D_1[1:r_1,1:r_1]$
\State $\B B = \tilde{\B A}_{22} - \tilde{\B A}_{12}^H\tilde{\B A}_{11}^{-1}\tilde{\B A}_{12}=\begin{bmatrix} \B B_{11} & \B B_{12} \\ \B B_{12}^H & 0\end{bmatrix}$ with $\B B_{11}\in\mathbb{F}^{\lfloor n/3\rfloor \times \lfloor n/3\rfloor}$
\If {$r_1 \geq n/3$}
\State $[\B P_2, \B L_{22}, \B D_2, r_2] = \text{Fast-LDL}(\B B)$
\State $\B P = \tilde{\B P_1}\begin{bmatrix} \B I & 0 \\ 0 & \B P_2\end{bmatrix}$,
$\B L = \begin{bmatrix} 
\B \tilde{L}_{11} &0\\
\B P_2^T (\tilde{D}_1^{-1}\tilde{L}_{11}^{-1}\tilde{\B A}_{12})^H & \B L_{22}
\end{bmatrix}$,
 $\B D = \begin{bmatrix} \tilde{D}_1 & 0 \\ 0 & \B D_2\end{bmatrix}$
\Else
\State $[\B P_B, \B Q_B, \B L_B, \B U_B, r_B] =  \text{Fast-LU}(\B B_{12})$
\State $\tilde{\B B} =\begin{bmatrix}\B B_{11} & \B B_{12}Q_B[:,1:r_B] \\ \B Q_B[:,1:r_B]^T\B B_{12}^H & 0\end{bmatrix}$
\State $[\B P_2, \B L_{22}, \B D_2, r_2] = \text{Fast-LDL}(\tilde{\B B})$
%\State $\B P = \begin{bmatrix} \B I_{r_1} & 0 & 0\\ 0 & \B P_2 & 0 \\ 0 & 0 & \B I_{n-\lfloor n/3 \rfloor -r_B} \end{bmatrix}\begin{bmatrix} \B I_{r_1+\lfloor n/3\rfloor}  \\  0 & \B P_B\end{bmatrix}\tilde{\B P_1}$
%%\textcolor{orange}{
\State $\B P = \tilde{\B P_1} \begin{bmatrix} \B I_{r_1+\lfloor n/3\rfloor} & 0 \\ 0 & \B Q_B\end{bmatrix}\begin{bmatrix} \B I_{r_1} & 0 & 0\\ 0 & \B P_2 & 0 \\ 0 & 0 & \B I_{n-r_1-\lfloor n/3 \rfloor -r_B} \end{bmatrix}$
%}
\State $\hat{\B A}=\B P^T \B A \B P = 
\begin{bmatrix} \hat{\B A}_{11} & \hat{\B A}_{12} & \hat{\B A}_{13}\\ 
\hat{\B A}_{12}^H& \hat{\B A}_{22}&  \hat{\B A}_{23} \\
\hat{\B A}_{13}^H & \hat{\B A}_{23}^H&  \hat{\B A}_{33}
\end{bmatrix}, 
\dim(\hat{\B A}_{11})=r_1, \dim(\hat{\B A}_{22})=\lfloor n/3\rfloor+r_B$
\State $\B L = \begin{bmatrix} \tilde{L}_{11} &0 & 0\\
L_{21} & \B L_{22} & 0 \\
L_{31} & (D_2^{+}\B L_{22}^{-1}(\hat{\B A}_{23}-L_{21}\tilde{D}_1L_{31}^H))^H & \B I
\end{bmatrix},$
\Statex where $L_{21} =(\tilde{D}_1^{-1}\tilde{L}_{11}^{-1}\hat{\B A}_{12})^H$, $L_{31} = (\tilde{D}_1^{-1}\tilde{L}_{11}^{-1}\hat{\B A}_{13})^H$, and $D_2^{+}$ is the pseudoinverse of $D_2$.
%\textcolor{orange}{
%\Statex
%$\B L = \begin{bmatrix} \B L_{11}[1:r_1,1:r_1] & 0\\
%(\B D_1[1:r_1,1:r_1]^{-1} \B L_{11}[1:r_1,1:r_1]^{-1} \hat{\B A}[1:r_1,r_1+1:])^H & 
%\begin{matrix}
%\B L_{22} &  0\\
%\begin{bmatrix}
%0 & (\B U_B[:,1:r_B]^{-1}U_B[:,r_B+1:])^H
%\end{bmatrix}\B P_2 \B L_{22} & \B I
%\end{matrix}
%\end{bmatrix}$
%}
\State $\B D = \begin{bmatrix} \tilde{D}_1 & 0 & 0 \\ 0 & \B D_2 & 0 \\ 0 & 0 & 0\end{bmatrix}$
\EndIf
\State $r=r_1+r_2$
\EndIf
\end{algorithmic}
\end{algorithm}

We now provide an algorithm (Algorithm~\ref{alg:ldl}) for LDL factorization with pivoting that runs in matrix-multiplication time.
The main challenge in designing a block-recursive algorithm for LDL is that the leading block on which we recurse may be low rank.
To circumvent this, we use recursive calls on matrices of dimension $2n/3$.
If such a block is near full rank, we make satisfactory progress in the recursive elimination, and otherwise, we may reduce the overall system into a saddle-point structure.
This algorithm uses LU with pivoting as a subroutine on off-diagonal blocks, much like in Algorithm~\ref{alg:ldl-nsp}.
%low-rank factorization of certain off-diagonal blocks.
%The recursive LDL subproblems have a saddle-point structure, and it would make sense to leverage 
We could alternatively leverage Algorithm~\ref{alg:ldl-nsp} directly, and this would not affect the asymptotic cost analysis.
\begin{theorem}
\label{thm:LDL_dense_fast}
Given $\B A \in \mathbb{F}^{n\times n}$, Algorithm~\ref{alg:ldl} computes the factorization $\B P^T \B A \B P = \B L \B D \B L^H$ and identifies the rank $r$ of $\B A$ using \sloppy
\[T_\text{LDL}(n)=O(n^\omega)\]
arithmetic operations.
\end{theorem}
\begin{proof}
The correctness of the algorithm can be verified based on the block equations of the LDL decomposition.
Each recursive step involves two recursive LDL calls on matrices of dimension $2n/3$ (this holds also when $r_1<n/3$, since $r_B \leq n/3$).
Additionally, an LU factorization, triangular inverses, and matrix multiplications are needed, all on matrices of dimension $O(n)$.
Hence, by Theorem~\ref{thm:LU_dense_fast} and bounds on cost of triangular inversion discussed in its proof, the cost of the LDL algorithm is
\begin{align*}
T_\text{LDL}(n)&=2T_\text{LDL}(2n/3) + O(n^\omega) = O(n^\omega).
\end{align*}
\end{proof}
The recursive nature of Algorithm~\ref{alg:ldl} is a challenge for integration with pivoting strategies such as that of Bunch and Kaufman~\cite{bunch1977some}.
A local version of any LDL pivoting strategy may be performed by applying the strategy in the base case, but this would not bound the growth factor in general.

\section{Sparse LDL Factorization}
\label{sec:sparse_ldl}

The sparse algorithm uses the same saddle-point reduction as Section~\ref{sec:nsldl}, but now it arises locally from a tree decomposition after eliminating vertices internal to leaf bags. The additional ingredient is peeling, which allows rank-deficient pieces to be propagated while preserving the treewidth bound.

The complexity of computing the Cholesky factorization of a sparse matrix without use of fast matrix multiplication is bounded by
\[O\bigg(\sum_{i=1}^n \nnz(\B \ell_i)^2\bigg),\]
where $\B \ell_i$ is the $i$th column in the $\B L$ factor.
The elimination width of a graph is the maximum value of $\nnz(\B \ell_i)$ over the columns of the Cholesky factor obtained from a given elimination ordering, and the minimum possible elimination width equals the treewidth~\cite{markov2008simulating,arnborg1985efficient,robertson1986graph,rose1970triangulated}.
Equivalently, $\tw(G)$ may be defined via tree decompositions of $G$~\cite{robertson1986graph}.
A tree decomposition of $G$ specifies a set of bags (subsets of vertices of $G$) and a tree over these bags, such that 
\begin{enumerate}
\item each vertex is assigned to at least one bag,
\item the set of bags to which any vertex is assigned is connected in the tree,
\item the end-points of any edge in $G$ are contained within some bag.
\end{enumerate}
The treewidth is the minimum, over tree decompositions, of the maximum bag size minus one.
Given a tree decomposition of width $\tau$, the standard multifrontal viewpoint yields an $O(n\tau^2)$ Cholesky algorithm, and replacing the dense frontal updates by matrix multiplication gives the familiar $O(n\tau^{\omega-1})$ bound.

%well known that the treewidth of a graph and cost of sparse matrix factorization are closely connected.

% \begin{definition}
% Let $G$ be a graph
% \end{definition}

After eliminating the vertices internal to a leaf bag, the remaining interaction with the rest of the tree has saddle-point form.
The sparse algorithm therefore applies the constraint-complemented partial LDL of Section~\ref{sec:nsldl} locally, but delays the corresponding edge eliminations until they can be executed without violating the treewidth bound.
The role of peeling is to remove linearly dependent rows of the local constraint block before those delayed updates propagate higher in the tree.

\subsection{Tree decomposition forms and transformations}
\label{sec:tree_forms}

Prior to introducing our sparse LDL algorithm, we recall a few useful results regarding transformations of tree decompositions.
In particular, we use that, for a graph with treewidth $\tau$, we can always find a tree decomposition that is a full binary tree and with $O(n/\tau)$ bags and bag size $O(\tau)$.

In general, for a tree decomposition $T$, let $b(T)$ denote its number of bags and, following notation in prior work, define
\[
\|T\|_p := \bigg(\sum_{\mathcal{B}\in T} |\mathcal{B}|^p\bigg)^{1/p},\qquad p>0.
\]

\begin{proposition}[Simplified form of {\cite[Theorem~13]{gosset2024fast}}]
\label{prop:td_few_bags}
If a graph on $n$ vertices has treewidth $t$, then it admits a rooted tree decomposition with $O(n/t)$ bags and bag size $O(t)$.
\end{proposition}
%\begin{proof}
%The cited theorem yields a stronger normal form with these same quantitative bounds.
%%Root the resulting tree arbitrarily.
%\end{proof}

\begin{proposition}[Simplified form of {\cite[Theorem~2]{
chatterjee2014optimal}}]
\label{prop:td_balanced}
Let $T$ be a rooted tree decomposition with $b(T)=b$ and maximum bag size $s$.
Then the same graph admits a rooted binary tree decomposition with $O(b)$ bags, maximum bag size $O(s)$, and height $O(\log b)$.
\end{proposition}
\begin{proof}
Apply the cited theorem to the given width-$(s-1)$ decomposition.
The resulting balanced binary tree decomposition has $O(b)$ bags, width $O(s)$, and height $O(\log b)$.
\end{proof}

\begin{corollary}
\label{cor:td_balanced_small}
If a graph on $n$ vertices has treewidth $\tau$, then it admits a rooted full binary tree decomposition with $O(n/\tau)$ bags and height $O(\log(n/\tau+1))$, while the maximum bag size is $O(\tau)$.
\end{corollary}
\begin{proof}
Apply Proposition~\ref{prop:td_few_bags}, then Proposition~\ref{prop:td_balanced}.
If some nodes are unary, add and connect leaf copies of the same bag.
This preserves the tree-decomposition property and bag sizes, changes the number of bags by at most a constant factor, and increases the height by at most $1$.
\end{proof}

We also note that Bodlaender and Hagerup give a similar result, except with a slightly looser bound on height ($O(\log n)$)
{\cite[Lemma~2.2]{BodlaenderHagerup1998}}.
Their analysis could also be combined with Proposition~\ref{prop:td_few_bags} to obtain the bounds in Corollary~\ref{cor:td_balanced_small}.

Restricting the formal algorithm to rooted full binary trees changes the parameters of the tree decomposition by at most a constant factor.
We therefore state the algorithm and the post-ordering formalism for rooted full binary trees.
This simplification allows a more concise algorithm pseudocode and analysis, but is not essential.

\subsection{LDL and LU Factorization with Peeling}

We consider a generalized LDL decomposition, which may be defined by introducing another primitive in addition to vertex and edge elimination (as described in Section~\ref{subsec:ldl_def}).
\begin{definition}[Vertex Peeling]
%Consider $A\in\mathbb{F}^{(n+1)\times (n+1)}$ and an associated graph $G$, where
%\[A = \begin{bmatrix} A_{11} & a_{21}^H \\ 
%                      a_{21}  & 0 \end{bmatrix}.\]
%If $\exists \B x\in\mathbb{F}^{n}$, such that $A_{11}\B x = a_{21}^H$, peeling the vertex in $G$ corresponding to the last row/column of $A$ is defined by the factorization,
%\[A = 
%\begin{bmatrix} I  \\  \B x^H  \end{bmatrix} A_{11}
%\begin{bmatrix} I  \\  \B x^H  \end{bmatrix}^H.\]
%We refer to the matrix $\begin{bmatrix} I  \\  \B x^H  \end{bmatrix}$ as a peeling transformation.
Consider $A\in\mathbb{F}^{(n+m+1)\times (n+m+1)}$ and an associated graph $G$, where
\[
A = \begin{bmatrix} A_{11} & A_{21}^H & a_{31}^H \\ 
                      A_{21} & 0 & 0 \\
                      a_{31} & 0 & 0 \end{bmatrix},
\]
with $A_{11}\in\mathbb{F}^{n\times n}$, $A_{21}\in\mathbb{F}^{m\times n}$, and $a_{31}\in\mathbb{F}^{1\times n}$.
If there exists $\B x\in\mathbb{F}^{m}$ such that $A_{21}^Hx = a_{31}^H$, then peeling the vertex in $G$ corresponding to the last row/column of $A$ is defined by the factorization,
\[A = 
      \begin{bmatrix} I  & 0 \\ 0 & I \\ 0 & \B x^H  \end{bmatrix} 
      \tilde{A}  
      \begin{bmatrix} I  & 0 \\ 0 & I \\ 0 & \B x^H  \end{bmatrix}^H, \qquad
      \tilde{A} = \begin{bmatrix} A_{11} & A_{21}^H  \\ A_{21} & 0  \end{bmatrix}.\]
We refer to $\begin{bmatrix} I  & 0 \\ 0 & I \\ 0 & \B x^H  \end{bmatrix}$ as a peeling transformation.
\end{definition}
%Supposing that $\tilde{A} = \begin{bmatrix} A_{11} & A_{21}^H  \\ A_{21} & 0  \end{bmatrix}$ is full rank, 
Given an LDL factorization of $\tilde{A}$, e.g., $\tilde{A}=\tilde{L}D\tilde{L}^H$, to compute a reduced LDL of $A$, we would only need the last row of $L=\begin{bmatrix} I  & 0 \\ 0 & I \\ 0 & \B x^H  \end{bmatrix}\tilde{L}$, which is given by $\begin{bmatrix} 0 & x^H\end{bmatrix}\tilde{L}$.
\begin{definition}[Peeled Implicit LDL Decomposition]\label{def:peel_ldl}
A peeled implicit LDL decomposition of a matrix $A$ is a sequence $Q_1,\ldots Q_m$ of permutations (reorderings of vertices), vertex eliminations, edge eliminations, and vertex peeling transformations that transforms the (reduced) $D\in\mathbb{F}^{\rank(A)\times\rank(A)}$ factor of an LDL decomposition of $A$ to $A$, i.e.,
\[A=\bigg(\prod_{i=1}^m Q_i\bigg)D\bigg(\prod_{i=1}^mQ_i\bigg)^H.\]
Further, $B=\prod_{i=1}^m Q_i$ is a permuted lower-trapezoidal matrix, i.e., $B=PL$ and a reduced LDL of $A$ is $A=PLDL^HP^T$.
\end{definition}
If $A=0$, the implicit LDL decomposition is composed of $\dim(A)$ vertex peeling transformations.
In general, the number of vertex peeling transformations in a peeled implicit LDL decomposition is $\dim(A)-\rank(A)$.
Given a peeled implicit LDL decomposition of $A\in\mathbb{F}^{n\times n}$, %with $k$ vertex peeling transformations,
we have a permutation $P$ and the explicit form of an LDL of the leading $\rank(A)\times \rank(A)$ full-rank block of $P^TAP$.
%Since $n-k\geq \rank(A)$, 
The peeling transformations may be used to obtain the full $L$ factor or to compute a product with $L$ implicitly.

Using a sequence of vertex/edge eliminations and peeling transformations, we will be able to ensure every vertex is connected to at most 2$\tau$ vertices when it or an edge it is connected to is eliminated, where $\tau$ denotes the treewidth parameter introduced below.
Ensuring this connectivity bound in turn bounds the number of nonzeros in the transformations composing the decomposition.
With use of only vertex and edge elimination, this is impossible, as illustrated by Example~\ref{ex:no_peel}.
\begin{example}\label{ex:no_peel}
Consider the adjacency matrix $A$ of the star graph (tree of height 1) with $n$ vertices, which has treewidth 1.
\begin{itemize}
\item An LDL of $A$ is defined by elimination of any edge in the graph.
This edge elimination eliminates the root vertex, resulting in an $L$ factor with $n-1$ nonzeros in one of its two columns.
\item
A peeled implicit LDL may be obtained by a sequence of $n-2$ peeling transformations on $n-2$ of the $n-1$ leaves of the star graph, followed by an edge elimination.
Each triangular or trapezoidal transformation in this decomposition has $1$ off-diagonal nonzero.
\end{itemize}
\end{example}
%mix of peeling as well as vertex and edge eliminations, we can define a sequence of transformations
%trans Peeling allows us to define a sequence of

Further, we may employ edge eliminations and vertex peeling transformations to compute an LU factorization.
To make the connection explicit, given a general rectangular matrix $B$, consider the $H$-symmetric matrix
\[A_B = \begin{bmatrix} 0 & B^H \\ B & 0\end{bmatrix}.\]
An edge elimination on $A_B$ corresponds to one elementary lower-triangular row operation on $B$, together with the matching upper-triangular column operation.
Similarly, peeling a vertex in $A_B$ corresponds to peeling a row or column of $B$.
Both transformations preserve the zero-diagonal block structure and block-sparsity of $A_B$.
\begin{definition}[Peeled Implicit LU Decomposition]
A peeled implicit LU decomposition of a matrix $B\in\mathbb{F}^{m\times n}$ is a sequence $(Y_1,Z_1),\ldots (Y_m, Z_m)$, such that each $(Y_i,Z_i)$ is one of the following,
\begin{itemize}
\item $Y_i$ and $Z_i$ are permutation matrices that act on rows and columns not peeled or eliminated by prior transformations,
\item $Y_i$ and $Z_i$ are elementary lower-triangular matrices with off-diagonal nonzeros only in the $j$th column, where $j-1$ is the number of prior elementary triangular transformations,
\item $Y_i$ is a peeling transformation and $Z_i=I$,
\item $Z_i$ is a peeling transformation and $Y_i=I$,
\end{itemize}
%pairs of permutations, pairs of elementary (lower- and upper-) triangular transformations, or a row peeling transformation and an identity matrix, 
such that there exists an LU factorization
\[P^TBQ = LU, \quad \text{where} \quad PL = \prod_{i=1}^m Y_i \text{ and }  QU^H=\prod_{i=1}^mZ_i.\]
\end{definition}
The elementary triangular transformations in a peeled implicit LU decomposition together give an LU factorization of the leading $\rank(B)\times \rank(B)$ block of $P^TBQ$ for some permutation matrices $P$ and $Q$.
The remaining part of the LU factorization of $B$ may be computed by multiplying the row peeling transformations with the elementary triangular transformations, or via triangular solves from the factors of the leading full-rank block of $P^TBQ$.
%, where $Q$ is the column permutation matrix.
\begin{lemma}\label{lem:ldl_to_lu}
Given a matrix $A \in \mathbb{F}^{(n+m)\times (n+m)}$ of the form,
\[A = \begin{bmatrix} 0 & B^H \\ B & 0 \end{bmatrix},\]
a peeled implicit LDL decomposition of $A$ 
%composed of only permutations, edge eliminations, and vertex peeling transformations 
yields a peeled implicit LU factorization of $B$.
\end{lemma}
\begin{proof}
Vertex eliminations on $A$ are impossible as its diagonal is zero, while any edge elimination preserves the same zero-diagonal block structure, hence any peeled implicit LDL decomposition of $A$ does not contain vertex elimination transformations.
Any edge elimination on $A$ must correspond to eliminating a nonzero in $B$ via elementary triangular transformations.
Similarly, vertex peeling transformations on $A$ may be restricted to row or column peeling transformations of $B$, since a vertex peeling transformation on $A$ acts (after permutation) on one of two types of rows, 
\begin{itemize}
\item a row of $B^H$, which must then be a linear combination of other rows of $B^H$,
\item a row of $B$, which must be a linear combination of other rows of $B$.
\end{itemize}
%Hence, the vertex peeling transformation can be expressed as a row or column peeling transformation of $B$.
%re is always an equivalent peeling transformation
%or a column peeling transformation on $B$.
%In particular, if the 
All permutations needed for edge eliminations and peeling transformations may be restricted to reorderings of rows and columns of $B$.
\end{proof}

\subsection{Sparse LDL Algorithm}

Given a tree decomposition with treewidth $\tau$, Cholesky admits an $O(n\tau^{\omega-1})$ algorithm by eliminating vertices assigned exclusively to leaf bags and performing dense updates on the resulting $O(\tau)\times O(\tau)$ frontal matrices.
We adopt the same block viewpoint for sparse LDL.
For a rooted tree decomposition, let $\rho(\mathcal{B})$ be the set of vertices
that would be eliminated in bag $\mathcal{B}$ by Cholesky, i.e., 
whose highest assigned bag in the rooted tree is $\mathcal{B}$.

By Corollary~\ref{cor:td_balanced_small}, we may transform any tree decomposition to a rooted full binary tree decomposition.
We therefore assume such a decomposition for the algorithmic formalism below.

\begin{definition}[Tree Decomposition Post-Ordering]
\label{def:tdpo}
Let $T$ be a rooted full binary tree decomposition.
A tree decomposition vertex post-ordering, denoted by $\prec$, is any total ordering obtained from a depth-first post-order traversal of the bags, with each visited bag $\mathcal{B}$ replaced by the block $\rho(\mathcal{B})$.
Equivalently, if $\mathcal{B}_X$ has children $\mathcal{B}_Y$ and $\mathcal{B}_Z$, then every vertex assigned to the subtree rooted at $\mathcal{B}_Y$ precedes every vertex assigned to the subtree rooted at $\mathcal{B}_Z$, and both subtrees precede the vertices of $\rho(\mathcal{B}_X)$.
The order within each block $\rho(\mathcal{B})$ is arbitrary.
\end{definition}
Definition~\ref{def:tdpo} is the only ordering property used in the sparse algorithm.

With a tree decomposition post-ordering, the vertices assigned exclusively to a leaf bag $\mathcal{B}_L$ with parent bag $\mathcal{B}_P$ are ordered first.
Here $\rho(\mathcal{B}_L)$ consists exactly of the vertices in $\mathcal{B}_L$ that are not also assigned to $\mathcal{B}_P$.
Let the sibling of this leaf bag be $\mathcal{B}_{S}$.
%%is also a leaf bag in the tree decomposition.
Then, we may consider the following blocking,
\[M = \begin{bmatrix} M_{11} & M_{21}^H \\ M_{21} & M_{22}\end{bmatrix},\]
where $M_{11}\in \mathbb{F}^{|\rho(\mathcal{B}_L)|\times |\rho(\mathcal{B}_L)|}$ and the number of nonzero rows in $M_{21}$ is at most $\tau$.
For Cholesky, we may factorize $M_{11}$, compute the Schur complement, and recurse on a tree decomposition without $\mathcal{B}_L$ and with $\tau$ treewidth.
In the general $H$-symmetric case, we may perform the LDL, $M_{11}=P_{11}L_{11}D_{11}L_{11}^HP_{11}^T$, but only eliminate a number of rows/columns equal to $r=\rank(M_{11})$.
After permutation, we may compute the Schur complement of the full rank block $(P_{11}^TM_{11}P_{11})[:r,:r]$, and reorder the blocks to obtain a saddle-point system,
\[M' = \begin{bmatrix} S & B^H \\ B & 0\end{bmatrix},\]
where $B\in\mathbb{F}^{(|\rho(\mathcal{B}_L)|-r)\times (n-|\rho(\mathcal{B}_L)|)}$ and $S\in \mathbb{F}^{(n-|\rho(\mathcal{B}_L)|)\times (n-|\rho(\mathcal{B}_L)|)}$.
%For example, if two leaf bags $\{1,2,3\}$ and $\{4,5,6\}$ share the parent separator $\{3,4\}$, then the induced post-ordering places the leaf-exclusive vertices $\{1,2\}$ and $\{5,6\}$ before the separator vertices $\{3,4\}$.
%After factorizing the leading block on one leaf, any rank deficiency appears in the off-diagonal block $B$, which couples the leftover rows to the separator.
%The role of peeling is to discard rows of this $B$ that are already dependent before they are allowed to interact with higher bags in the tree.
%Prior to apply the edge complementations in $F$, 

To compute the LDL factorization of the above saddle-point system while preserving treewidth, we can leverage the Schilders factorization approach to the constraint-complemented partial LDL factorization.
However, performing complementation of $B$ naively may result in eliminating an edge adjacent to a vertex that is connected to all other vertices, and in doing so we may create a dense graph and ruin the tree decomposition.
To circumvent this, as in Algorithm~\ref{alg:ldl-nsp}, we may compute the LU factorization up front to identify a set of edge eliminations, and any linearly dependent rows in $B$.
Further, rather than compute the Schur complement explicitly, we may delay edge eliminations until vertices lower in the tree are eliminated first.

However, delaying edge eliminations creates other issues.
Until we perform $\rank(B)$ edge eliminations associated with rows of $B$, other edge eliminations from different parts of the tree will in general affect $B$ (the rows of $L$ associated with $B$) and create connections between vertices that previously did not share a bag.
This second challenge is resolved by use of peeling, which allows elimination (peeling) of linearly dependent rows of $B$ up-front, ensuring the number of delayed edge eliminations at any vertex in the tree is at most $\tau$.
Example~\ref{ex:peel2} illustrates what could potentially go wrong if we do not peel linearly dependent rows.
\begin{example}\label{ex:peel2}
Given a graph $G$ with adjacency matrix $A=\begin{bmatrix} A_{11} & A_{21}^H \\ A_{21} & 0\end{bmatrix}$, where $A_{11}$ and $A_{21}$ are square and full rank, consider the matrix,
\[B =
\begin{bmatrix}
A_{11} & A_{21}^H & I \\
A_{21} & 0 & 0 \\
I & 0 & 0 
\end{bmatrix}.\]
%where $A_{21}$ is square.
%Assume $A_{12}$ is full rank, so $\rank(B)=\rank(A)$, and 
Clearly, $B$ is the adjacency matrix of a graph with at most double the treewidth of $G$.
We could peel rows/columns of $B$ to reduce it to $A$.
If we instead simply factorize $A=LDL^H$, with $A_{11}=L_{11}D_1L_{11}^H$, the full $L$ factor of $B$ will have an off-diagonal block equal to $L_{11}^{-H}D_1^{-1}$, the upper-triangular part of which is in general dense.
\end{example}

Hence, to compute the LDL for $M'$, we first eliminate vertices internal to $\mathcal{B}_S$ or its descendants, and use vertex peeling to merge the delayed edge eliminations from $\mathcal{B}_S$ with $B$, producing a new off-diagonal block, $B'$.
Then, we apply any delayed edge eliminations with edges connected to $\rho(\mathcal{B}_P)$ (until we have zeroed out the first $|\rho(\mathcal{B}_P)|$ columns of $B'$).
After this, we proceed to eliminate (via vertex eliminations, edge eliminations, delayed edge eliminations, and peeling) the vertices in $\rho(\mathcal{B}_P)$, etc..

\begin{algorithm}[!hp]
\caption{$[\{Q_i\}_{i=1}^m, S, B]=$ tree-LDL($\B A$, $T$, $\gamma$)}\label{alg:spldl}
\begin{algorithmic}[1]

\Require $\B A\in\mathbb{F}^{n\times n}$, $\B A^H = \B A$, $\B A$ has an associated graph $G=(V,E)$, which has a tree decomposition described by bags $\mathcal{B}_1,\ldots,\mathcal{B}_k\subset V=\{1,\ldots,n\}$ connected by binary full tree $T$ with root $\mathcal{B}_X$.
$\gamma$ is a non-negative integer, $\gamma \leq |\mathcal{B}_X|$.
The rows and columns of $\B A$ are ordered according to a tree decomposition post-ordering.
\Ensure $\{Q_i\}_{i=1}^m$ is a sequence of transformations as in Definition~\ref{def:peel_ldl}, giving a partial peeled implicit LDL decomposition of the form,
\[\begin{bmatrix}A_{11} & A_{12} \\ A_{21} & A_{22} \end{bmatrix} =\bigg(\prod_{i=1}^m Q_i\bigg) \begin{bmatrix} S & B^H \\ B & 0\end{bmatrix}\bigg(\prod_{i=1}^m Q_i\bigg)^H.\]
with $S\in \mathbb{F}^{(n-\gamma)\times (n-\gamma)}$ corresponding to the same block as $A_{22}$ (i.e., the last $n-\gamma$ rows and columns of $A$ are not eliminated or re-ordered internally by $\{Q_i\}_{i=1}^m$).  $B$ is full rank.
%$m=|\mathcal{B}_X\setminus\rho(\mathcal{B}_X)|$, 
%$(\ell_2,\ldots,\ell_k)$-lower-trapezoidal.
%$\hat{L}_\perp$ includes the columns of $L_\perp$, and 
%$\hat{L}_\perp^H$ is in row-echelon form and
%$U$ is upper-trapezoidal.

%$L$, $D$, and $P$ give a partial LDL factorization such that 
%\[P^T\begin{bmatrix}A & \bar{L}_\perp \\ \bar{L}_\perp^H &0\end{bmatrix}P-LDL^H = \begin{bmatrix} 0 & 0 & 0 & 0\\ 0 &S & \hat{L}_{\perp} & L_X \\ 0 & \hat{L}_{\perp}^H & 0 & 0\\ 0 &  L_X^H & 0 & 0\end{bmatrix},\]
%where $\bar{L}_\perp = \begin{bmatrix} 0 \\ \bar{L}_\perp \end{bmatrix}$, $S\in |\mathcal{B}_X\setminus\mathcal{B}_X^{(1)} |\times|\mathcal{B}_X\setminus\mathcal{B}_X^{(1)}|$ and $\hat{L}_\perp =\begin{bmatrix} L_2&  \cdots & L_k\end{bmatrix}$.
%%where $S\in |\mathcal{B}_X\setminus\mathcal{B}_X^{(1)} |\times|\mathcal{B}_X\setminus\mathcal{B}_X^{(1)}|$ and $\hat{L}_\perp =\begin{bmatrix} L_2&  \cdots & L_k\end{bmatrix}$.
\If {$k=1$}
\State Return $[\{Q_i\}_{i=1}^m, S, B]=$ tree-LDL-substep($\B A$, $\emptyset$, $\gamma$)
\EndIf
\State Let $\mathcal{B}_Y$ and $\mathcal{B}_Z$ be the children of $\mathcal{B}_X$, and let $T^{(Y)}$ and $T^{(Z)}$ be the subtrees of $T$ with roots $\mathcal{B}_Y$ and $\mathcal{B}_Z$. Let $y$ and $z$ be the total number of vertices in all bags of $T^{(Y)}$ and $T^{(Z)}$ excluding those in $\mathcal{B}_X$.
$A$ has the block structure,
\[\begin{bmatrix}
A_{11} & 0 & A_{13} & A_{14} & 0 & 0\\
0 & A_{22} &0 & A_{24} & A_{25} & 0 \\
A_{31} &0 & A_{33} & A_{34} & A_{35} & A_{36}\\
A_{41} & A_{42} & A_{43} & A_{44} & A_{45} & A_{46}\\
0 & A_{52} & A_{53} & A_{54} & A_{55} & A_{56} \\
0 &0 & A_{63} & A_{64} & A_{65} & A_{66}
\end{bmatrix}, \quad
\underbrace{\begin{bmatrix}
 A_{33} & A_{34} & A_{35} & A_{36}\\
 A_{43} & A_{44} & A_{45} & A_{46}\\
A_{53} & A_{54} & A_{55} & A_{56} \\
 A_{63} & A_{64} & A_{65} & A_{66}
% A_{33} & A_{34} & A_{35}\\
% A_{43} & A_{44} & A_{45} \\
% A_{53} & A_{54} & A_{55}
\end{bmatrix}}_{A^{(X)}} \in \mathbb{F}^{|\mathcal{B}_X|\times |\mathcal{B}_X|},\]
$A_{11}\in\mathbb{F}^{y\times y}$, 
$A_{22}\in\mathbb{F}^{z \times z}$,
$A_{31}\in\mathbb{F}^{|\mathcal{B}_X\cap \mathcal{B}_Y\setminus \mathcal{B}_Z|\times y}$,
$A_{41}\in\mathbb{F}^{|\mathcal{B}_X\cap \mathcal{B}_Y\cap \mathcal{B}_Z|\times y}$,
$A_{52}\in\mathbb{F}^{|\mathcal{B}_X\cap \mathcal{B}_Z\setminus \mathcal{B}_Y|\times z}$.
\State 
\(A^{(Y)} = 
\begin{bmatrix}
A_{11} & A_{13} & A_{14} \\
A_{31} & 0 & 0 \\
A_{41} & 0 & 0 
\end{bmatrix}
\),
\(A^{(Z)} = 
\begin{bmatrix}
A_{22} & A_{24} & A_{25} \\
A_{42} & 0 & 0 \\
A_{52} & 0 & 0 
\end{bmatrix}
\) 
\State $[\{Q_i^{(Y)}\}_{i=1}^{{m_Y}}, S^{(Y)}, B^{(Y)}]=$ tree-LDL($\B A^{(Y)}$, $T^{(Y)}$, $|\mathcal{B}_X \cap \mathcal{B}_Y|$)
\State $[\{Q_i^{(Z)}\}_{i=1}^{{m_Z}}, S^{(Z)}, B^{(Z)}]=$ tree-LDL($\B 
A^{(Z)}$, $T^{(Z)}$, $|\mathcal{B}_X \cap \mathcal{B}_Z|$)
\label{li:rec2}
\State $S^{(X)} = A^{(X)} + 
        \begin{bmatrix} 
        S^{(Y)}_{11} & 0 & S^{(Y)}_{12} & 0 \\
        0 & 0 & 0 & 0  \\
        S^{(Y)}_{21} & 0 & S^{(Y)}_{22} & 0  \\
        0 & 0 & 0 & 0 
        \end{bmatrix} 
        +
        \begin{bmatrix} 
        0 & 0 & 0 & 0 \\
        0 & S^{(Z)}_{11} &  S^{(Z)}_{12} & 0 \\
        0 & S^{(Z)}_{21} &  S^{(Z)}_{22} & 0 \\
        0 & 0 & 0 & 0 
        \end{bmatrix} 
        $
\State $B^{(X)} = \begin{bmatrix} B^{(Y)}_1 & 0 & B^{(Y)}_2 & 0 \\ 0 & B^{(Z)}_1 & B^{(Z)}_2 & 0  \end{bmatrix}$
\State $[\{Q_i^{(X)}\}_{i=1}^{m_X}, S, B]=$ tree-LDL-substep($\B S^{(X)}$, $B^{(X)}$, $\gamma$)
\State Obtain $\{Q_i\}_{i=1}^{m}$ from 
$\{Q_i^{(X)}\}_{i=1}^{m_X}$,
$\{Q_i^{(Y)}\}_{i=1}^{m_Y}$, and
$\{Q_i^{(Z)}\}_{i=1}^{m_Z}$.
\State Return $\{Q_i\}_{i=1}^{m}, S, B$
\end{algorithmic}
\end{algorithm}

%\begin{definition}[$(\ell_1,\ldots,\ell_k)$-lower-trapezoidal]
%For some integers $k\geq 1$ and $\ell_1,\ldots, \ell_k\geq 0$, with $\sum_{i=1}^k\ell_i=n$, we say a matrix $L\in\mathbb{F}^{n\times m}$ is $(\ell_1,\ldots,\ell_k)$-lower-trapezoidal if 
%$L_\perp = \begin{bmatrix} L_1&  \cdots & L_k\end{bmatrix}$, where each $L_i$, has at most $\ell_i$ columns, is zero in the first $\sum_{j=1}^{i-1}\ell_j$ rows and lower-trapezoidal in the remaining block.
%\end{definition}

\begin{algorithm}[!h]
\caption{$[\{Q_i\}_{i=1}^m, S, F]=$ tree-LDL-substep($\B A$, $B$, $\gamma$)}\label{alg:spldl_step}
\begin{algorithmic}[1]
\Require $\B A\in\mathbb{F}^{n\times n}$, $\B A^H = \B A$. $\B B\in\mathbb{F}^{k\times n}$, $k\geq 0$.
\Ensure $\{Q_i\}_{i=1}^m$ is a sequence of transformations as in Definition~\ref{def:peel_ldl}, giving a partial peeled implicit LDL decomposition of the form,
\[\begin{bmatrix}A_{11} & A_{12} & B_1^H\\ A_{21} & A_{22} & B_2^H \\
B_1 & B_2 & 0\end{bmatrix} = \bigg(\prod_{i=1}^m Q_i\bigg)\begin{bmatrix} S & F^H \\ F & 0\end{bmatrix}\bigg(\prod_{i=1}^m Q_i\bigg)^H.\]
with $S\in \mathbb{F}^{(n-\gamma)\times (n-\gamma)}$ corresponding to the same block as $A_{22}$ (i.e., the last $n-\gamma$ rows and columns of $A$ are not eliminated or re-ordered internally by $\{Q_i\}_{i=1}^m$).  $F$ is full rank.
%
%\[P^T\begin{bmatrix}A & L_\perp \\ L_\perp^H &0\end{bmatrix}P-LDL^H = \begin{bmatrix} 0 & 0 &  0\\
%0 &W &  \hat{L}_{\perp}U \\ 
%0 &  (\hat{L}_{\perp}U)^H  & 0\end{bmatrix},\]
%where $W\in \mathbb{F}^{(n-\ell)\times (n-\ell)}$, $\hat{L}_\perp$ includes the last $n-\gamma$ columns of $L_\perp$, and $\hat{L}_\perp^H$ is in row-echelon form.
%%$(\ell_2,\ldots,\ell_k)$-lower-trapezoidal.
%$U$ is upper-trapezoidal.
\State Assume $B$ is full rank. (Peel any linearly dependent columns of $B^H$ in $\begin{bmatrix} A & B^H \\ B & 0 \end{bmatrix}$).
\State Perform constraint complementation of $B_1$ in $\begin{bmatrix}A_{11} & A_{12} & B_1^H\\ A_{21} & A_{22} & B_2^H \\
B_1 & B_2 & 0\end{bmatrix}$, resulting in Schur complement,
\[Y = \begin{bmatrix} Y_{11} & Y_{12} & 0 \\ Y_{12}^H & \tilde{A}_{22} & \tilde{B}_2^H \\ 0 &\tilde{B}_2 & 0\end{bmatrix}, \]
where $Y_{11}\in\mathbb{F}^{(\gamma-\rank(B_1))\times (\gamma-\rank(B_1))}$, while $\tilde{A}_{22}$ and $\tilde{B}_2$ are the same shape as $A_{22}$ and $B_2$.
\State Compute the LDL of $Y_{11}$ and eliminate $\ell=\rank(Y_{11})$ rows and columns in $Y$, resulting in a new Schur complement,
\[Z = \begin{bmatrix} 0 & Z_{12} & 0 \\ Z_{12}^H & S & \tilde{B}_2^H \\ 0 &\tilde{B}_2 & 0\end{bmatrix},\]
where $Z_{12}\in\mathbb{F}^{(\gamma-\rank(B_1)-\rank(Y_{11}))\times (n-\gamma)}$.
\State Peel vertices corresponding to any linearly dependent columns in $\begin{bmatrix}\tilde{B}_2^H & Z_{12}^H\end{bmatrix}$ in corresponding columns of $Z$ by computing LU factorization of this matrix. Let $F^H$ be the remaining, linearly independent, subsequence of columns of $\begin{bmatrix}\tilde{B}_2^H & Z_{12}^H\end{bmatrix}$. Permute the columns in $F^H$ to be after $S$ in the Schur complement and similar with rows.
\State Define $m$, $\{Q_i\}_{i=1}^m$ based on the eliminations, permutations, and peeling steps performed above.
%\[S'=\begin{bmatrix} G & F^H \\ F & 0 \end{bmatrix},\]
%with $G\in\mathbb{F}^{(n-\ell)\times(n-\ell)}$.
%\State Form $L$ and $D$ based on the eliminations (Schur complementation, LDL, constraint complementation).
%\State Compute the LU factorization of 
%$\begin{bmatrix} L_\perp^{(2)} &F^H\end{bmatrix}=\hat{L}_\perp UQ^T$ via Algorithm~\ref{alg:lup} with $\hat{L}_\perp$ in reduced row echelon form.
%%and subsequent permutation of rows to ensure each block of $\hat{L}_\perp$, i.e., each set of columns with the index of the leading nonzero entry in the interval $\Big(\sum_{i=2}^{j} \ell_i,\sum_{i=2}^{j+1} \ell_i\Big]$ for $j\in\{1,\ldots,k-1\}$, is lower-trapezoidal.
%\State Let $W=P_\perp^T G P_\perp$.
%\State Form $P$ based on the permutations performed in all steps above.
\end{algorithmic}
\end{algorithm}

The full recursive algorithm is given in Algorithm~\ref{alg:spldl}, which leverages the subroutine defined by Algorithm~\ref{alg:spldl_step}.
Algorithm~\ref{alg:spldl} is written for rooted full binary trees to simplify the block notation at one recursive step.
If a bag has one child, the same local step is obtained by omitting the missing child contribution; if it has more than two children, recurse on each child subtree, concatenate the returned full-rank constraint blocks in any fixed child order, and apply the same local substep to the resulting block matrix.

In the following theorem, we bound the cost of Algorithm~\ref{alg:spldl}, focusing on the arithmetic operations performed.
Initial reordering and permutations throughout the algorithm would entail an additional cost that is independent of $\tau$ and at most $O(n\log n)$.
\begin{theorem}
\label{thm:LDL_sparse_fast}
Let $T$ be a rooted full binary tree decomposition of the graph $G=(V,E)$ associated with the off-diagonal part of $\B A\in\mathbb{F}^{n\times n}$, $A^H=A$.
Algorithm~\ref{alg:spldl}, interpreted on $T$ as above, computes a peeled implicit LDL decomposition of $\B A$ with
\[
O(\|T\|_\omega^\omega)=O\bigg(\sum_{\mathcal{B}\in T} |\mathcal{B}|^\omega\bigg)
\]
arithmetic operations over $\mathbb{F}$.
For each bag $\mathcal{B}$, the transformations produced at $\mathcal{B}$ act on $O(|\mathcal{B}|)$ rows and columns, and each elimination or peeling transformation has $O(|\mathcal{B}|)$ off-diagonal nonzeros.
\end{theorem}
\begin{proof}
At a bag $\mathcal{B}$, the matrix passed to Algorithm~\ref{alg:spldl_step} is supported on the rows and columns indexed by $\mathcal{B}$ together with the full-rank constraint blocks returned by its two child subproblems.
Because each such constraint block is full rank and its columns are indexed by vertices in $\mathcal{B}$, it has at most $|\mathcal{B}|$ rows.
Hence every dense matrix appearing in the local substep has dimension $O(|\mathcal{B}|)$.

The cost of Algorithm~\ref{alg:spldl_step} at bag $\mathcal{B}$ is the sum of four dense kernels on matrices of dimension $O(|\mathcal{B}|)$:
\begin{itemize}
\item constraint complementation of $B_1$, cost $O(|\mathcal{B}|^{\omega})$ by Algorithm~\ref{alg:ldl-nsp},
\item LDL factorization of $Y_{11}$, cost $O(|\mathcal{B}|^{\omega})$ by Theorem~\ref{thm:LDL_dense_fast},
\item LU factorization used to peel linearly dependent columns of $\begin{bmatrix}\tilde{B}_2^H & Z_{12}^H \end{bmatrix}$, cost $O(|\mathcal{B}|^{\omega})$ by Theorem~\ref{thm:LU_dense_fast},
\item Schur-complement and triangular-update operations, cost $O(|\mathcal{B}|^{\omega})$.
\end{itemize}
Summing over the bags of $T$ gives the stated complexity bound.

The same local support argument yields the sparsity bound.
Vertex eliminations at $\mathcal{B}$ involve only vertices in $\mathcal{B}$.
Constraint complementation and peeling involve only the rows and columns represented in the same local saddle-point block.
Thus each elimination or peeling transformation produced at $\mathcal{B}$ has $O(|\mathcal{B}|)$ off-diagonal nonzeros.
\end{proof}

\begin{corollary}
\label{cor:LDL_sparse_treewidth}
Suppose the graph of $\B A$ has treewidth $\tau$.
Then $\B A$ admits a peeled implicit LDL decomposition that can be computed in $O(n\tau^{\omega-1})$ arithmetic operations.
If $D_{\mathrm{dense}}(s)$ bounds the parallel depth of the dense $s\times s$ kernels used in Algorithm~\ref{alg:spldl_step}, then there is a balanced full binary hierarchy for which the sparse recursion has parallel depth
\[
O(\log(n/\tau+1)\,D_{\mathrm{dense}}(O(\tau))).
\]
\end{corollary}
\begin{proof}
By Corollary~\ref{cor:td_balanced_small}, there is a rooted full binary tree decomposition $T$ with $O(n/\tau)$ bags, bag size $O(\tau)$, and height $O(\log(n/\tau+1))$.
Then
\[
\|T\|_\omega^\omega = O((n/\tau)\tau^\omega)=O(n\tau^{\omega-1}).
\]
The cost bound follows from Theorem~\ref{thm:LDL_sparse_fast}.

The computation at a bag depends only on its children, so bags at the same tree depth may be processed independently.
Each level performs dense kernels on $O(\tau)\times O(\tau)$ blocks, which gives the stated depth bound.
\end{proof}

The peeled implicit LDL of a sparse matrix allows us to obtain the full LDL factorization if the matrix is close to full rank, as well as an LU factorization.
\begin{corollary}
The LDL factorization of a sparse matrix $A\in\mathbb{F}^{n\times n}$, $A^H=A$, whose off-diagonal nonzeros correspond to a graph with treewidth $\tau$ can be computed with cost $O(n\tau^{\omega-1})$ if $n-\rank(A)=O(\tau)$.
\end{corollary}
\begin{proof}
By Corollary~\ref{cor:LDL_sparse_treewidth}, the peeled implicit LDL of $A$ may be computed with cost $O(n\tau^{\omega-1})$.
This peeled LDL gives a permutation $P$ and an explicit LDL of the leading, full-rank, $\rank(A)\times \rank(A)$ block of $P^TAP$.
The remaining part of the LDL decomposition may be computed by triangular solve with $O(n-\rank(A))$ right-hand sides, which has cost $O(n\tau^{\omega-1})$ so long as the number of right-hand sides is bounded by $O(\tau)$.
\end{proof}
\begin{corollary}
The peeled implicit LU factorization of a sparse matrix $B\in\mathbb{F}^{m\times n}$, for which the bipartite graph with adjacency matrix,
\[A = \begin{bmatrix} 0 & B^H \\ B & 0\end{bmatrix},\]
has treewidth $\tau$, can be computed with cost $O((m+n)\tau^{\omega-1})$.
Further, the full LU factorization may be obtained in the same complexity if $\max(m,n)-\rank(B)=O(\tau)$.
\end{corollary}
\begin{proof}
By Corollary~\ref{cor:LDL_sparse_treewidth}, the peeled implicit LDL of $A$ may be computed with cost $O((m+n)\tau^{\omega-1})$ after applying it to $A=\begin{bmatrix} 0 & B^H \\ B & 0\end{bmatrix}$.
By Lemma~\ref{lem:ldl_to_lu}, this peeled LDL gives permutation $P$ and $Q$ and an explicit LU of the leading, full-rank, $\rank(B)\times \rank(B)$ block of $P^TBQ$.
The remaining parts of the LU decomposition may be computed by triangular solve with $O(\max(m,n)-\rank(B))$ right-hand sides, which has cost $O((m+n)\tau^{\omega-1})$ so long as the number of right-hand sides is bounded by $O(\tau)$.
\end{proof}
Additionally, given a tree decomposition $T$ of the bipartite graph, by Theorem~\ref{thm:LDL_sparse_fast} we can compute the LU of $B$ with $O(\|T\|_\omega^{\omega})$ arithmetic operations.

\input{inverse_butterfly_section}

\section{Conclusion}
%Section~\ref{sec:nsldl} gives an LDL formulation of the null-space approach for saddle-point systems, and Section~\ref{sec:app_ldl_strassen} places dense LU and dense LDL within the same fast-matrix-multiplication framework.
%The main result is the sparse factorization theorem of Section~\ref{sec:sparse_ldl}, which shows that bounded-treewidth sparse symmetric elimination can be organized around saddle-point reductions, null-space factorizations, and peeling while retaining the target complexity and fill bounds.

For an $n\times n$ matrix $A$ corresponding to a graph with treewidth $\tau$, this paper provides $O(n\tau^{\omega-1})$ algorithms to identify the rank of $A$, to compute its inertia (if symmetric and real or complex), to obtain a low-rank factorization in implicit form, or to compute a full LDL or LU factorization of $A$ if it is near full rank.
Section~\ref{sec:inverse_butterfly} also shows that for a full-rank matrix with bounded-treewidth, its inverse is complementary low rank and admits a butterfly factorization with rank $O(\tau)$.
A recent independent development following this work is the treewidth-based LU algorithm of F\"urer, Hoppen, and Trevisan~\cite{furer2025fast}.
Our work differs in considering LDL, fast matrix multiplication, and leveraging insights from saddle-point systems, but both approaches are similar in that only the leading square full-rank block can be factorized explicitly without additional cost.
Hence, when $A$ is low rank or somewhat low rank (e.g., $\rank(A)=n/2$), the problem of computing the explicit form of the non-trivial lower-trapezoidal part of the $L$ factor
%in a low rank factorization
%(when $A=A^H$, finding $U\in\mathbb{F}^{n\times\rank(A)}$ such that $A=UU^H$, and, in general, finding $U,V\in\mathbb{F}^{n\times\rank(A)}$ such that $A=UV^H$)
appears to be harder\footnote{It is also unclear to us whether $O(n\tau^2)$ complexity is possible for low rank LDL and LU.}.
For LU, the same challenge arises in factorization of rectangular sparse matrices.
Efficient algorithms for this problem or lower bounds on its complexity are an important subject for further study.
The utility of the proposed implicit LDL and LU decompositions is also of interest, e.g., while the peeled implicit LDL factorization may not give us the factor $L$, it allows us to compute $L^HX$ for any $X\in\mathbb{F}^{n\times \tau}$ in $O(n\tau^{\omega-1})$ time via application of blocks of peeling or elimination transformations.

%when encountering rank-deficient submatrices.
%For sparse matrices, naive (as opposed to implicit) application of vertex and edge elimination in LDL does not appear to suffice to attain the correct complexity and minimize fill.

Our complexity bounds focus on achieving $\omega$-dependent cost, but the proposed algorithms also benefit from reducing work to large matrix-matrix products, for example through improved communication efficiency.
Additional open questions remain with regard to the stability of the proposed algorithms and their integration with pivoting strategies.
For sparse factorizations, we would argue that minimizing fill, as achieved by the algorithm in Section~\ref{sec:sparse_ldl}, is higher priority than minimizing pivot growth.
Even when round-off error results in an approximate factorization, the factorization could be used for preconditioning, e.g., via direct approximation of the matrix or (in the context of saddle point systems) via constraint preconditioning~\cite{keller2000constraint}.

Our paper also builds a direct general connection between low-treewidth and low-rank butterfly factorization.
Extension of this result to consider (hierarchically) semi-separable matrices would also be of interest.

\section{Acknowledgments}
We are grateful to Yuchen Pang for numerical verification of the dense LU and LDL algorithms over GF(2), as well as Claude-Pierre Jeannerod for helpful comments on an earlier version of the manuscript.
This research was supported by the United States Department of Energy (DOE) Advanced Scientific Computing Research program via award DE-SC0023483.
\bibliographystyle{siamplain}
\bibliography{references,inverse_butterfly_refs}

\appendix
\section{Dense LU Factorization}
\label{app:lu}

\begin{algorithm}[h]
\caption{$[\B P,\B Q,\B L,\B U,r]=$ Fast-LU($\B A$)}\label{alg:lup}
\begin{algorithmic}[1]
\Require $\B A\in\mathbb{F}^{m\times n}$
\Ensure $\B P^T\B A\B Q  = \B L \B U $, $\B L\in\mathbb{F}^{m\times r}$  is lower-trapezoidal and unit-diagonal, $\B U\in\mathbb{F}^{r\times n}$ is upper-trapezoidal, $\B P$ and $\B Q$ are permutation matrices, $r$ is the rank of $\B A$. $(P^TL)^H$ is in row-echelon form.
\If {$m=1$}
 \If {$\B A=0$}
  \State Set $\B P=[1]$, $\B Q=I_n$, $\B L\in\mathbb{F}^{1\times 0}$, $\B U\in\mathbb{F}^{0\times n}$, and $r=0$.
 \Else
  \State Pick $\B Q$ to permute the leading nonzero of $\B A$ to the 1st column and set $\B P=[1]$, $\B L=[1]$, $\B U=\B A\B Q$, and $r=1$.
 \EndIf
\Else 
\State $\B A = \begin{bmatrix}\B A_1 \\ \B A_2\end{bmatrix}$, $\B A_1\in\mathbb{F}^{\lfloor m/2\rfloor \times n}$
\State $[\B P_1, \B Q_1, \B L_1, \B U_1, r_1] = \text{Fast-LU}(\B A_1)$
\State Compute $\B B_1 = (\B A_2\B Q_1)[:,1:r_1]\B U_1[:,1:r_1]^{-1}$
\State Compute $\B B_2 = (\B A_2\B Q_1)[:,r_1+1:] - \B B_1\B U_1[:,r_1+1:]$
\State $[\B P_2, \B Q_2, \B L_2, \B U_2, r_2] = \text{Fast-LU}(\B B_2)$
%\State Compute $\B C = \B L_2[:,1:r_2]^{\#}\B B[:,r_1+1:n]$
\State $\B P = \begin{bmatrix} \B P_1[:,1:r_1] & 0 & \B P_1[:,r_1+1:] \\ 0 & \B P_2 & 0\end{bmatrix}$
\label{li:alg:lup:p}
\State 
$\B Q = \B Q_1\begin{bmatrix}\B I & 0 \\ 0 & \B Q_2\end{bmatrix}$
\State $\B L = \begin{bmatrix}\B L_1[1:r_1,:] & 0  \\
\B P_2^T\B B_1 & \B L_2[:,1:r_2] \\
\B L_1[r_1+1:,:] & 0
\end{bmatrix}$
\State $\B U = \begin{bmatrix} \B U_1[:,1:r_1] &  \B U_1[:,r_1+1:]\B Q_2 \\
0 &\B U_2 \end{bmatrix}$
\State $r=r_1+r_2$
\EndIf
\end{algorithmic}
\end{algorithm}

Algorithm~\ref{alg:lup} provides a recursive fast-matrix-multiplication-based algorithm for rectangular dense LU factorization of the form $P^TAQ=LU$, where $P$ and $Q$ are permutation matrices.
The algorithm recurses on subsets of rows of $A$, hence it cannot be combined with row-wise partial pivoting or to perform complete pivoting, but $Q$ may be constructed based on partial pivoting.
%We state the algorithm for GF(2), but given recursion is column-wise only, performing partial pivoting (for pivoted LU over $\mathbb{R}$ or $\mathbb{C}$) would only change the $n=1$ (base-case) step in the recursion.

%\begin{algorithm}[t]
%\caption{$[\B P,\B L,\B U,R]=$ GF(2)-LU($\B A$)}\label{alg:lup}
%\begin{algorithmic}
%\Require $\B A\in\mathbb{F}_2^{m\times n}$ 
%\Ensure $\B A \B P = \B L \B U \bmod 2$, $\B L$ lower-triangular, $\B U$ is unit-diagonal and upper-trapezoidal, $R$ is the rank of $\B A$
%\If {$n=1$}
% \State Factorize $\B A\B P = \B L \B U \bmod 2$ directly, determine $R\in\{0,1\}$
%\Else 
%\State $\B A = [\B A_1, \B A_2]$, $\B A_1\in\mathbb{F}_2^{m\times \lfloor n/2\rfloor}$
%\State $[\B P_1, \B L_1, \B U_1, r_1] = \text{GF(2)-LU}(\B A_1)$
%\State Compute $\B B = \B L_1[:,1:r_1]^{\#}\B A_2$
%\State $[\B P_2, \B L_2, \B U_2, r_2] = \text{GF(2)-LU}(\B B[r_1+1:m,:])$
%%\State Compute $\B C = \B L_2[:,1:r_2]^{\#}\B B[:,r_1+1:n]$
%\State $\B P = \begin{bmatrix} \B P_1[:,1:r_1]& \B P_2[:,1:r_2]& \B P_1[:,r_1+1:]& \B P_2[:,r_2+1:]\end{bmatrix}$
%\State $\B L = \begin{bmatrix}\vert & 0 &0 \\
%\B L_1[:,1:R_1] &\B L_2 & 0 \\ \vert & \vert & \B I\end{bmatrix}$
%\State $\B U = \begin{bmatrix}\B U_1[1:R_1,1:R_1] & \B B[1:R_1,:] & \B U_1[1:R_1,R_1+1:] \\
%0 & \B U_2[:,1:R_1] & 0 \end{bmatrix}$
%\State $R=R_1+R_2$
%\EndIf
%\end{algorithmic}
%\end{algorithm}

\begin{theorem} (Theorem~\ref{thm:LU_dense_fast})
\label{thm:LU_dense_fast2}
\sloppy
Given $\B A \in \mathbb{F}^{m\times n}$, Algorithm~\ref{alg:lup} computes the factorization $\B P^T\B A\B Q = \B L \B U$ and identifies the rank $r$ of $\B A$ using
\[T_\text{LU}(m,n,r)=O(r^{\omega-2}mn)=O(\min(m,n)^{\omega-1}\max(m,n))\]
arithmetic operations.
Further, $(P^TL)^H$ is in row-echelon form. 
\end{theorem}
\begin{proof}
The correctness of the algorithm can be verified based on the block equations of the LU decomposition.
At each recursive step, the cost of the algorithm is dominated by inverting an $r_1\times r_1$ block of $\B U$ and computing matrix-matrix products to form $\B B_1$ and $\B B_2$.
%\todo{clarify pseudo inverse} Edgar: Its an inverse
%Given $\B L_1[:,1:r_1]^{\#}\in\mathbb{F}_2^{n\times O(\min(m,n))}$, 
These products may be decomposed into at most $O(mn/r_1^2)$ products of $r_1\times r_1$ blocks.
%matrix-matrix product can be computed with cost $O(R^{\omega-2}mn)$ by subdividing $\B A_2$ into blocks of size $r_1\times r_1$.
The inverse of a triangular matrix may be computed by recursively inverting on diagonal blocks, followed by matrix multiplication.
%Hence, we can form $\B L_1[:,1:r_1]^{\#}=\begin{bmatrix} \B L_1[1:r_1,1:r_1]^{\#} \\ \B L_1[r_1+1:,1:r_1]^{\#}\B L_1[1:r_1,1:r_1]^{\#}\end{bmatrix}$ by inverting the triangular $r_1\times r_1$ leading block with cost $O(R^{\omega})$, followed by a matrix product of cost $O(R^{\omega-1}m)$.
Overall, we have
\begin{align*}
T_\text{LU}(m,n,r)&\leq \max_{r_1,r_2,r_1+r_2=r}T_\text{LU}(m/2,n,r_1) + T_\text{LU}(m/2,n,r_2)+ O(r^{\omega-2}mn)\\
&=O(r^{\omega-2}mn).
\end{align*}
Further, $(P^TL)^H$ is in row-echelon form, since rows found to be linearly dependent are moved to the end at each recursive step, while pivots appear in strictly increasing column positions.
\end{proof}
For improved numerical stability, it may be advisable to augment Algorithm~\ref{alg:lup} to perform a more sophisticated row-wise pivoting strategy such as pairwise pivoting~\cite{Sorensen_1676570,Tiskin2007179}, tournament pivoting~\cite{Demmel:2010}, or a randomized algorithm~\cite{dong2023simpler}.

\end{document}

%% file: inverse_butterfly_section.tex
\section{Inverse butterfly factorization from sparse LDL}
\label{sec:inverse_butterfly}

In this section we assume that $A$ is full rank, so that Section~\ref{sec:sparse_ldl} yields an explicit sparse $LDL^H$ factorization
\[
P^TAP = LDL^H.
\]
We use this factorization to derive rank bounds for separator blocks of $A^{-1}$ and a sparse product factorization of $A^{-1}$.
Classical elimination-based inverse representations and exact structured inverse classes already show that Gaussian elimination can encode inverse application and inverse structure compactly \cite{ursic1982inverse,Meurant1992,FasinoGemignani2002,ChandrasekaranDewildeGuPalsSunVanderVeenWhite2005}.
In particular, it is well-known that the inverse of a tridiagonal matrix is a semi-separable matrix with rank-1 off-diagonal blocks~\cite{FasinoGemignani2002}.
We generalize this notion to bounded treewidth.
We prove exact separator-rank bounds and show that the resulting inverse factorization is a standard butterfly factorization.
For background on butterfly factorization and related sparse-factor constructions see \cite{LiYangMartinHoYing2015,LiYangYing2018,YangLiYing2017,PangHoYang2020,LiuXingGuoMichielssenGhyselsLi2021}.
%By Corollary~\ref{cor:td_balanced_small}, we assume the tree decomposition of $A$ is a rooted full binary tree decomposition with $O(n/\tau)$ bags, bag size $O(\tau)$, and height $h=O(\log n)$. 
%\begin{theorem}[Bodlaender--Hagerup \cite{BodlaenderHagerup1998}]
%\label{thm:bh-balancing}
%If a graph on $n$ vertices has a tree decomposition of width $k$, then it also has a tree decomposition of width at most $3k+2$ and height $O(\log n)$.
%\end{theorem}

\subsection{Inverse complementary low rank based on treewidth}

%For a bag $\mathcal B$ in a tree decomposition $T$ and a connected component $\mathcal C$ of $T-\mathcal B$, define
%\[
%\Omega_{\mathcal B}(\mathcal C)=\left(\bigcup_{\mathcal D\in \mathcal C}\mathcal D\right)\setminus \mathcal B.
%\]
%For fixed $\mathcal B$, the sets $\Omega_{\mathcal B}(\mathcal C)$ are pairwise disjoint by the running-intersection property of the tree decomposition.
%We write $\sqcup$ for disjoint union.

The complementary low-rank property is commonly used to define matrices that admit a butterfly decomposition.
To define this property, generally one assumes the rows/columns of a matrix are associated with a binary tree of height $h=O(\log n)$.
Then, the property implies that the rank associated with the part of a matrix that describes interactions between a subtree at a level $l$ and another at level $h-l$ is low-rank.
We define a stricter version of this property, which we show is satisfied for a rank of $O(\tau)$ given the inverse of a matrix of treewidth $\tau$.
\begin{definition}[Strict complementary rank]
\label{def:separator-clr}
Consider a symmetric matrix $K\in\F^{n\times n}$ and associated graph $G$.
$K$ has strict complementary rank $r$, if there exists a one-to-one mapping between vertices in $G$ and a binary tree, such that for any node $u$ in the binary tree, the subtree $S$ rooted at $u$ and the remaining nodes in the tree $V$ satisfy
\[\rank(K_{S,V})\leq r,\]
where $K_{S,V}\in\mathbb{F}^{|S|\times |V|}$ is the off-diagonal block of a permutation of $K$ where vertices in $S$ are ordered before vertices in $V$.
%for some permutation matrix $P$, and blocking $b_1,\ldots, b_k$, where $k=O(n/r)$ and $\sum_{i=1}^k b_i = n$, any off-diagonal block $A_{21}\in\mathbb{F}^{n_1\times n_2}$, where for some $l<k$, $n_1= \sum_{i=1}^l b_i$, and $n_2=\sum_{i=l+1}^k b_i$, $\rnk(A_{21})\leq r$.
% with respect to a tree decomposition $T$ if for every bag $\mathcal B$ of $T$ and every partition
%\[
%\mathfrak C_1\sqcup \mathfrak C_2
%\]
%of the connected components of $T-\mathcal B$ into two nonempty families, the index sets
%\[
%U_s=\bigcup_{\mathcal C\in \mathfrak C_s}\Omega_{\mathcal B}(\mathcal C),\qquad s\in\{1,2\},
%\]
%satisfy
%\[
%\rnk\bigl(K_{U_1,U_2}\bigr)\le r.
%\]
%The pairwise formulation, in which each $\mathfrak C_s$ consists of a single component, is the special case obtained by taking singleton families.
\end{definition}

\begin{theorem}[Treewidth-based complementary rank of the inverse]
\label{thm:inverse-separator-complementary}
Let $A\in\F^{n\times n}$ be invertible and $A^H=A$. 
Let the graph of $A$ have treewidth $\tau$.
Then $A^{-1}$ has strict complementary rank $O(\tau)$.
% exact separator complementary low rank of width $O(\tau)$ with respect to $T$. More precisely, for every bag $\mathcal B$ of $T$ and every partition $\mathfrak C_1\sqcup\mathfrak C_2$ of the connected components of $T-\mathcal B$ into two nonempty families,
%\[
%\rnk\bigl((A^{-1})_{U_1,U_2}\bigr)\le |\mathcal B|,
%\]
%where $U_s=\bigcup_{\mathcal C\in\mathfrak C_s}\Omega_{\mathcal B}(\mathcal C)$.
\end{theorem}

\begin{proof}
By Corollary~\ref{cor:td_balanced_small}, there exists a full binary tree decomposition $T$ of the graph of $A$ with width $O(\tau)$ and height $O(\log(n/\tau + 1))$
Associate an $n$-node binary tree $\hat{T}$ with the rows/columns of $A$ by mapping the vertices eliminated in any bag, $\rho(\mathcal{B})$, to a chain (path graph) and connecting the root of any children of $\mathcal{B}$ in $T$ to the leaf in this chain (this also yields a valid elimination tree for $A$).
Consider any vertex $u \in\hat{T}$, the vertices $S$ in the subtree rooted at $u$, and the remainder of the vertices $V$.
Let $\mathcal{B}$ be the tree decomposition bag in $T$ that satisfies $u\in \rho(\mathcal{B})$.

Now, add a unary duplicate $\mathcal{B}'$ of $\mathcal{B}$ as the parent of $\mathcal{B}$, and re-root the tree so that $\mathcal{B}'$ is the root, yielding a binary tree decomposition $T'$.
Assume wlog that $A$ is ordered according to a tree decomposition post ordering for $T'$.
Hence if $|S\setminus \mathcal{B}| = n_1$, $|V\setminus \mathcal{B}| = n_2$, and $|\mathcal{B}|=n_3$, we have that
\[A = \begin{bmatrix}A_{11} & 0 & A_{13} \\ 0 & A_{22} & A_{23} \\ A_{31} & A_{32} & A_{33}\end{bmatrix},\]
where $A_{ii}\in\mathbb{F}^{n_i\times n_i}$ for $i\in\{1,2,3\}$.
Now, consider the $L$ factor in the LDL of $A$,
\[L = \begin{bmatrix}L_{11} & 0 & 0 \\ 0 & L_{22} &  0 \\ L_{31} & L_{32} & L_{33}\end{bmatrix}.\]
We have that
\begin{align*}
A^{-1} =& L^{-H}D^{-1}L^{-1} = \begin{bmatrix} L_{11}^{-H}D_1^{-1}L_{11}^{-1} & 0 & 0 \\ 0 & L_{22}^{-H}D_2^{-1}L_{22}^{-1} &0 \\ 0 & 0 & 0\end{bmatrix} + U^HD_3^{-1}U, \\
& U =  \begin{bmatrix} -L_{33}^{-1}L_{31}L_{11}^{-1} & -L_{33}^{-1}L_{32}L_{22}^{-1} & L_{33}^{-1} \end{bmatrix}.
\end{align*}
Hence, the off-diagonal block in $A^{-1}$ connecting $S$ and $V$ is of rank at most $\rnk(U) = O(\tau)$.
The strict complementary rank is hence bounded by $O(\tau)$ since permuting $O(\tau)$ rows/columns in $\mathcal{B}$ to the end changes the rank of the off-diagonal block considered in Definition~\ref{def:separator-clr} by at most $|\mathcal{B}|\leq \tau$.

\end{proof}

%The same factorization also shows that one may adjoin the separator vertices $S=\mathcal B$ to either side of the bipartition without changing the bound $|\mathcal B|$: the resulting off-diagonal block still factors through $G_S$. We state Theorem~\ref{thm:inverse-separator-complementary} on the sets $V\setminus S$ because this isolates the complementary part of the inverse associated with the separator and gives the cleanest exact nonasymptotic bound.
%
%For the product factorization we now fix the balanced rooted full binary decomposition. As in Section~\ref{sec:sparse_ldl}, for each bag $\mathcal B$ let $\rho(\mathcal B)$ denote the set of vertices assigned to $\mathcal B$. We write $\mathcal C\preceq \mathcal B$ when $\mathcal C$ is a descendant of $\mathcal B$ in the rooted tree, possibly with $\mathcal C=\mathcal B$, and define
%\[
%I(\mathcal B)=\bigcup_{\mathcal C\preceq \mathcal B}\rho(\mathcal C).
%\]
%If $\mathcal B$ has children $\mathcal B_L$ and $\mathcal B_R$, write
%\[
%I_L=I(\mathcal B_L),\qquad I_R=I(\mathcal B_R),\qquad S=\rho(\mathcal B),
%\]
%so that
%\[
%I(\mathcal B)=I_L\sqcup I_R\sqcup S.
%\]
%The separator-rank theorem above is the exact complementary low-rank property that underlies the butterfly factorization: at an internal bag, the two child index sets $I_L$ and $I_R$ interact only through the separator block $S$.

\subsection{Butterfly decomposition of low-treewidth matrix inverse}

Following~\cite{LiYangMartinHoYing2015}, we define the butterfly decomposition of a matrix, which captures off-diagonal low-rank structure in a tree-based view of the rows and columns of a matrix.
\begin{definition}[Butterfly decomposition]
Given an $H$-symmetric matrix $K$ of strict complementary rank $r$, a butterfly decomposition of $K$ is 
\[K=B_1\cdots B_k,\]
where $k=O(\log n)$ and each $B_i$ consists of a diagonal part and at most $\lceil n/r\rceil$ nonzero blocks of size $r\times r$.
We refer to $r$ as the rank of the butterfly decomposition.
\end{definition}
Consistent with prior literature, we do not restrict the position of the nonzero blocks in each $B_i$, but to the best of our understanding, the block-structure of the factorization we obtain is essentially the same as in papers introducing the butterfly decomposition~\cite{LiYangMartinHoYing2015}.

\begin{theorem}
Consider an invertible $H$-symmetric matrix $A\in\mathbb{F}^{n\times n}$ associated with graph $G$.
If $G$ has treewidth $\tau$, $A^{-1}$ admits a butterfly factorization of rank $O(\tau)$.
\end{theorem}
\begin{proof}
Recall the form of $A^{-1}$ derived in the proof of Theorem~\ref{thm:inverse-separator-complementary},
\begin{align*}
A^{-1} =& \begin{bmatrix} L_{11}^{-H}D_1^{-1}L_{11}^{-1} & 0 & 0 \\ 0 & L_{22}^{-H}D_2^{-1}L_{22}^{-1} &0 \\ 0 & 0 & 0\end{bmatrix} + U^HD_3^{-1}U, \\
& U =  \begin{bmatrix} -L_{33}^{-1}L_{31}L_{11}^{-1} & -L_{33}^{-1}L_{32}L_{22}^{-1} & L_{33}^{-1} \end{bmatrix}.
\end{align*}
We can alternatively express $L^{-1}$ as
\begin{align}
L^{-1}
&=
\begin{bmatrix}
I & 0 & 0  \\
0 & I & 0\\
0 & 0 &L_{33}^{-1}
\end{bmatrix}
\begin{bmatrix}
I & 0 & 0  \\
0 & I & 0\\
-L_{31} & -L_{32} & I 
\end{bmatrix}
\begin{bmatrix}
L_{11}^{-1} & 0 & 0  \\
0 & L_{22}^{-1} & 0\\
0 &  0 &I
\end{bmatrix}\label{eq:tril}.
\end{align}
Now, assume that the bag $\mathcal{B}$ associated with vertices in $L_{33}$, is the root of $T$, and $T$ is a complete binary tree decomposition (we may obtain a complete tree with height $O(\log(n/\tau+1))$ by copying bags for any missing nodes in a full binary tree).
Then, we have that $L_{11}$ and $L_{22}$ are factors obtained from a tree decomposition of the two child subtrees of $\mathcal{B}$ and hence admit the same recursive structure as $L$.
Hence, we may express
\begin{align*}
\begin{bmatrix}
L_{11}^{-1} & 0 & 0  \\
0 & L_{22}^{-1} & 0\\
0 &  0 &I
\end{bmatrix} &= \begin{bmatrix} R_1V_1Z_1 & 0 & 0 \\ 0 & R_2V_2Z_2 & 0 \\ 0 & 0 & I\end{bmatrix} \\
&= 
\begin{bmatrix} R_1 & 0 & 0 \\ 0 & R_2 & 0 \\ 0 & 0 & I\end{bmatrix}
\begin{bmatrix} V_1 & 0 & 0 \\ 0 & V_2 & 0 \\ 0 & 0 & I\end{bmatrix}
\begin{bmatrix} Z_1 & 0 & 0 \\ 0 & Z_2 & 0 \\ 0 & 0 & I\end{bmatrix},
\end{align*}
where $R_i$, $V_i$, and $Z_i$ are defined in the same way as \eqref{eq:tril} for each $i$.
The matrices $R_i$ and $V_i$, for each $i$, have $O(n\tau)$ nonzeros, present in $O(n/\tau)$ blocks.
Then, $Z_1$ and $Z_2$ themselves are block-diagonal with four inverse $L$ factors, associated with four subtrees, rooted at the grand-children of the root.
Hence, we may recursively expand them to obtain $L$ as a product of $2h+1$ block-sparse matrices (for a tree of height $h$), each with at most $O(n\tau)$ nonzeros.
Inserting this expansion into $A^{-1}=L^{-H}D^{-1}L^{-1}$, we have obtained a butterfly decomposition.
\end{proof}
The cost and parallel depth of obtaining this butterfly decomposition is the same as needed for LDL, since the butterfly factors are obtained directly from the entries of $L$, plus the inversion of $O(\tau)\times O(\tau)$ diagonal blocks of $L$.
This connection can also be extended to nonsymmetric matrices, by the same embedding we use to reduce LU to LDL.